\documentclass[a4paper,10pt,hidelinks]{article}
\usepackage{graphicx} 
\usepackage[a4paper, margin=2 cm]{geometry}
\usepackage{hyperref}
\usepackage{amsthm}
\usepackage{thm-restate}
\usepackage{geometry}
\usepackage{amssymb}
\usepackage{amsmath}
\usepackage{cleveref}
\usepackage{soul}
\usepackage{comment}
\usepackage{setspace}
\usepackage{xcolor}
\usepackage[shortlabels]{enumitem}
\usepackage{listings}
\usepackage{blkarray}
\usepackage{lmodern}
\usepackage[english]{babel}
\usepackage{float}
\usepackage{tikz}
\usetikzlibrary{arrows.meta,arrows}
\usetikzlibrary{arrows,decorations.markings}
\usetikzlibrary{decorations.pathmorphing}
\usepackage[square,sort,comma,numbers]{natbib}
\newtheorem{theorem}{Theorem}[section]
\newtheorem{lemma}[theorem]{Lemma}
\newtheorem{corollary}[theorem]{Corollary}

\newtheorem{claim}[theorem]{Claim}
\newtheorem*{claim*}{Claim}
\newtheorem{proposition}[theorem]{Proposition}
\newtheorem{conjecture}[theorem]{Conjecture}
\newtheorem{remark}[theorem]{Remark}

\theoremstyle{definition}
\newtheorem{definition}[theorem]{Definition}
\newtheorem*{definition*}{Definition}

\newcommand{\eps}{\varepsilon}

\newcommand{\AM}[1]{\textcolor{red}{\textsf{#1}}}

\newcounter{propcounter}

\title{The Lov\'asz conjecture holds for moderately dense Cayley graphs}
\author{Benjamin Bedert \thanks{University of Cambridge, UK. The author gratefully acknowledges financial support from the European Research Council (ERC) Starting Grant “High Dimensional Probability and Combinatorics”, grant No. 101165900. \emph{Email:} \textbf{bb741@cam.ac.uk} {\nolinkurl{bb741@cam.ac.uk}}.}\and Nemanja Dragani\'c\thanks{Mathematical Institute, University of Oxford. \emph{Email:} \textbf{nemanja.draganic@maths.ox.ac.uk} } \and Alp M\"uyesser\thanks{New College, University of Oxford. \emph{E-mail:} \textbf{alp.muyesser@new.ox.ac.uk}.} \and Mat\'ias Pavez-Sign\'e\thanks{Department of Mathematical Engineering and Center for Mathematical Modeling  (CNRS IRL2807), University of Chile. Supported by ANID Fondecyt Regular grant No. 1241398 and ANID Basal Grant CMM FB210005.  \emph{E-mail:} \textbf{\ mpavez@dim.uchile.cl}.}}

\date{}

\begin{document}

\maketitle

\begin{abstract}
We show that there is an absolute constant $c>0$ such that every large connected $n$-vertex Cayley graph with degree $d\ge n^{1-c}$ has a Hamilton cycle. This makes progress towards the Lov\'asz conjecture and improves upon the previous best result of this form due to Christofides, Hladk\'y, and M\'ath\'e from 2014 concerning graphs with $d\geq \eps n$. Our proof avoids the use of Szemer\'edi's regularity lemma and relies instead on an efficient arithmetic regularity lemma specialised to Cayley graphs. 


\end{abstract}
\section{Introduction}
The Hamilton cycle problem is one of the most fundamental and widely studied questions in computer science and combinatorics. It was one of Karp's original 21 NP-complete problems (see~\cite{Karp1972}) and determining whether a general graph contains a Hamilton cycle remains a central challenge in the theory of algorithms and complexity. While the problem is computationally intractable in the worst case, a guiding philosophy in the field is that sufficient structure should render the problem solvable. In particular, it is a classical intuition that high symmetry might be enough to force the existence of a Hamilton cycle.

This intuition is formalized in a celebrated conjecture of Lovász~\cite{lovaszhamcycle} from 1969, which asserts that every connected vertex-transitive graph has a Hamilton cycle. An even older version of this problem, posed by Rapaport-Strasser~\cite{rapapport} in 1959, focuses on Cayley graphs. Given a group $G$ and a subset $S\subset G$, the \textit{Cayley graph} $\mathrm{Cay}_G(S)$ is the graph whose vertices are the elements of $G$ and all the edges are of the form $\{g,gs\}$ with $g\in G$ and $s\in S$. We will assume throughout the paper that $G$ is finite and $S$ is symmetric (i.e. $s^{-1}\in S$ whenever $s\in S$), in which case $\mathrm{Cay}_G(S)$ is indeed a graph rather than a directed graph.

\begin{conjecture}\label{conj:lovasz}Every connected Cayley graph on a finite group with at least 3 elements is Hamiltonian.\end{conjecture}

Despite over sixty years of research, Conjecture 1.1 remains wide open.
It has been verified for abelian groups and specific families such as $p$-groups (for references and many more results see~\cite{kutnar2009hamilton}). 
Pak and Radoičić~\cite{PAK20095501} showed that every finite group $G$ of size at least 3 has a generating set $S\subset G$ with $|S|\le \log_2|G|$ for which $\mathrm{Cay}_G(S)$ is Hamiltonian. The conjecture is also known for ``typical'' generating sets: if $S\subset G$ is a random subset of size $|S|\geq C\log |G|$ (where $C$ is a sufficiently large constant), then with high probability $\mathrm{Cay}_G(S)$ is an expander graph (due to the Alon--Roichman theorem~\cite{alon1994random}), and hence, is Hamiltonian, as established in a recent result~\cite{draganic2024hamiltonicity}.

\par For arbitrary Cayley graphs without typicality assumptions, Hamiltonicity is far from being established. In fact, there is no widespread consensus on whether the conjecture is true; indeed, Babai~\cite{babai1994automorphism} conjectured that the statement is false, suggesting the existence of infinitely many connected Cayley graphs $G$ where the longest cycle has length at most $(1-c)|G|$ for a fixed constant $c>0$. In the other direction, the longest cycle that is known to exist has length $\Theta(n^{9/14})$~\cite{norin2026small} (see also (\cite{groenland2025longest, bucic2026long}). 

\par On the other hand, for dense Cayley graphs, meaning those Cayley graphs with $\Omega(n^2)$ edges, the Lovasz conjecture holds. This was established by Christofides, Hladk{\`y} and M{\'a}th{\'e}~\cite{christofides2014hamilton} who showed that for every $\varepsilon>0$ and $n$ sufficiently large, every $n$-vertex connected Cayley graph\footnote{In fact, the result of \cite{christofides2014hamilton} applies to the broader class of vertex-transitive graphs.} of degree $d\ge \varepsilon n$ has a Hamilton cycle. 
A key part of the argument in \cite{christofides2014hamilton} is a structural result that decomposes a given vertex-transitive graph into components that satisfy a robust notion of connectivity. The proof in \cite{christofides2014hamilton} uses Szemer\'edi's regularity lemma to facilitate this decomposition. \par The regularity lemma is a powerful tool, however, it does not give useful information for sparse graphs, i.e. those with $o(n^2)$ edges. The use of the regularity lemma also yields dependencies that are of tower-type, presenting a challenge for practical efficiency. Although the methods of Christofides, Hladk{\`y} and M{\'a}th{\'e} have influenced further results that also provide a structural decomposition of graphs into their expanding components while managing to avoid the use of the regularity lemma (see the work of K\"uhn, Lo, Osthus and Staden~\cite{kuhn2015robust} and also \cite{kuhn2014hamilton}), such results also come with costly exponential dependencies and do not apply to sparse graphs. 

\par In this paper, we break this density barrier encountered in previous approaches and confirm the Lov\'asz conjecture for polynomially sparse Cayley graphs.
\begin{theorem}\label{thm:mainthm} There exists a constant $c\geq1/200$ such that for all sufficiently large $n$, every connected $n$-vertex Cayley graph of degree $d > n^{1-c}$ has a Hamilton cycle. 
\end{theorem}

A key feature of our approach is that it avoids the machinery of the regularity lemma, relying instead on mild forms of spectral expansion. 
One of the key ingredients in our proof is a weak arithmetic regularity lemma (see Theorem~\ref{prop:RobustCayley-general} below) recently introduced by Bedert, Buci\'c, Kravitz, Montgomery and M\"uyesser~\cite[Theorem 1.5]{bedert2025graham} in their work concerning Graham's rearrangement conjecture. This weak arithmetic regularity lemma is a structural decomposition result, similar in spirit to the aforementioned work, but with much better quantitative dependencies. However, unlike those provided by the usual regularity lemma, the expansion conditions guaranteed by weak arithmetic regularity lemma are quite mild. Therefore, using this decomposition to produce a Hamilton cycle is a delicate task and is the crux of our work. To surmount this difficulty, we develop an efficient implementation of the absorption method specialised to sparse Cayley graphs. See Section~\ref{sec:outline} for a proof overview.

\subsection{Future directions}
\par We made no effort to optimize the constant $1/200$ that appears in Theorem~\ref{thm:mainthm}, however, our methods as presented here cannot be pushed beyond Cayley graphs of degree $\sqrt{n}$. Beyond this threshold, the weak arithmetic regularity lemma stops giving useful information and several of our arguments that build cycles out of small linear forests start to break down.
Further progress requires a more efficient regularity lemma providing a decomposition into \textit{approximate subgroups} rather than subgroups (see the discussion in~\cite{bedert2025graham}) and also further ideas for finding Hamilton cycles in sparser graphs. Both of these aspects present considerable technical challenges, so we leave confirming the Lov\'asz conjecture for graphs of degree $n^{o(1)}$ for future work. 

\par It is likely that dense Cayley graphs might have enough structure to support Hamilton cycles, or even Hamilton decompositions, in a more robust sense than general Cayley graphs. For example, Alspach~\cite{alspach1990decomposition} conjectured that all abelian Cayley graphs of even degree admit a decomposition of their edge-set into Hamilton cycles. This is known to not hold for general (nonabelian) Cayley graphs \cite{bryant2015vertex}, but, under a density assumption as in the current paper, all even degree Cayley graphs may admit a Hamilton decomposition. In another direction, Erd\H{o}s and Trotter \cite{trotter1978cartesian} showed that a directed version of the Lov\'asz conjecture cannot hold (see \cite{bucic2026long} for further interesting recent work in this direction), however, it appears to us that in a dense regime, a directed generalisation may hold, and our methods may generalise to this setting. 
\par Finally, we put forth the following conjecture that would be a simultaneous generalisation of the Lov\'asz conjecture and Pos\'a's seminal result that the random graph $G(n,p)$ is Hamiltonian for $p\gg \log n /n$.
\begin{conjecture}\label{conj:percolation}
    There exists a sufficiently large $C$ so that the following holds. Let $G$ be a Cayley graph on $n$ vertices with degree $d$ and let $G_p$ be the random graph obtained by sampling each edge of $G$ independently with probability $p$. If $p\geq C\log n/d$, then with high probability, $G_p$ has a Hamilton cycle.
\end{conjecture}
It would already be interesting to confirm the above when $d>n^{1-c}$, and we believe our techniques would bear relevance for this problem. We also remark that if the above conjecture holds, in particular, $G_p$ would contain a perfect matching (if $n$ is even). Determining whether this is true already seems to be an interesting problem.

Motivated by algorithmic applications for the perfect matching problem, Goel, Kapralov, and Khanna~\cite{goel2010perfect, goel2019perfect} (see also \cite{glebov2021perfect}) studied the analogous problem when $G$ is a general $d$-regular bipartite graph. In this case, they showed that $G_p$ contains a perfect matching with high probability whenever $p\geq cn\log n/d^2$, and they provided constructions showing the bound on $p$ is best possible up to logarithmic factors. Our conjecture above, on the other hand, predicts that perfect matchings (and Hamilton cycles) in Cayley graphs are significantly more robust to random edge-deletions compared to arbitrary regular graphs. 

\par We remark that in the context of Conjecture~\ref{conj:percolation}, we can at least ensure that $G_p$ is connected with high probability. Indeed, any (connected) $d$-regular Cayley graph $G$ is $d$-edge-connected~\cite{godsil2013algebraic}, and a well-known result of Karger~\cite{karger1994using} implies that for such $G$, $G_p$ is $\Omega(pd)$ edge connected as long as $p=\Omega(\log n/d)$.

\section{Outline of the proof}\label{sec:outline}

Suppose $G$ is an $n$-element group and let $S\subset G$ be a symmetric generating subset with $|S|=\sigma n$, where $\sigma\ge n^{-c}$. A general strategy for finding a Hamilton cycle in $\mathrm{Cay}_G(S)$ can be summarised in the following four steps.
\begin{enumerate}[label=\upshape{\textbf{Step~\arabic{enumi}}}]
    \item\label{RRS:1} Find an absorbing structure $A$ in $\mathrm{Cay}_G(S)$ that can later incorporate a ``small'' set of leftover vertices.
    \item\label{RRS:2} Find a linear forest $\mathcal F$ in $\mathrm{Cay}_G(S)-A$ using as few paths as possible and covering almost every vertex in $\mathrm{Cay}_G(S)-A$.
    \item\label{RRS:3} Link up $A$ and the paths in $\mathcal F$ to form an almost spanning cycle $C$ which has $A$ as a segment.
    \item\label{RRS:4} Use the absorbing property of $A$ to incorporate the remaining vertices to complete a Hamilton cycle. 
\end{enumerate}
This general strategy has been used with great success to attack problems in various settings when the host graph is dense (see e.g. \cite{rodl2008approximate}). For sparse graphs, however, implementing such a strategy presents major difficulties and has only been accomplished in sparse graphs with random-like behaviour (see e.g.~\cite{glock2024hamilton,randomspanningtree}). 
To make this strategy work in the Cayley setting,  we will need to address a series of issues that are not present in the dense case or the random/pseudorandom case, as we outline below.

\subsection{First reduction}As in~\ref{RRS:3} we need to link up components from a linear forest, some sort of robust connectivity property may be helpful. Although it is not true that arbitrary Cayley graphs have the connectivity properties we would like, we will use the fact that they can be partitioned into cosets with strong connectivity properties, as defined below.

\begin{definition}A $\zeta$-sparse-cut in a graph $\Gamma$ is a partition $V(\Gamma)=X\cup Y$ such that $e(X,Y)\leq \zeta |X||Y|$.\end{definition}
We will use the following result from~\cite{bedert2025graham}.
\begin{theorem}[Bedert, Buci\'c, Kravitz, Montgomery, M\"uyesser~\cite{bedert2025graham}]\label{prop:RobustCayley-general}
Let $\sigma\in(0,1]$ and $\varepsilon\in(0,1/2)$. Let $G$ be a finite group and $S\subset G$ be a subset with density $\sigma$. Then, there is a subgroup $H\leq G$ such that
    \begin{enumerate}[label=\upshape{(\roman{enumi})}]
        \item\label{reg:lemma:1} $|S\cap H|\geq (1-\varepsilon)|S|$, and
        \item\label{reg:lemma:2} $\mathrm{Cay}_H(S\cap H)$ has no $\varepsilon\sigma^3/1000$-sparse cuts.
    \end{enumerate}    
\end{theorem}
Theorem~\ref{prop:RobustCayley-general} can be seen as some sort of \textit{weak arithmetic regularity lemma} for Cayley graphs, similar to Green's arithmetic regularity lemma~\cite{green2005szemeredi}. We remark that~\ref{reg:lemma:2} is a consequence of a more general fact about the spectrum of $\mathrm{Cay}_H(S\cap H)$. Indeed, Theorem 1.5 from~\cite{bedert2025graham} ensures that there is a small but significant gap between the first two eigenvalues of the adjacency matrix of $\mathrm{Cay}_H(S\cap H)$, which, together with a version of the \textit{expander mixing lemma}, implies that $\mathrm{Cay}_H(S\cap H)$ satisfies~\ref{reg:lemma:2}.

 For $\eps>0$, Theorem~\ref{prop:RobustCayley-general} gives a subgroup $H\le G$ such that $|S\cap H|\ge (1-\eps)|S|$ and $\mathrm{Cay}_H(S\cap H)$ has no $\eps\sigma^3/1000$-sparse cuts, and thus, as $G$ partitions into left cosets $g_1H,g_2H,\ldots, g_mH$, we obtain a (vertex) partition of $\mathrm{Cay}_G(S)$ into graphs which are all isomorphic to $\mathrm{Cay}_H(S\cap H)$, and therefore with no $\eps\sigma^3/1000$-sparse cuts (here, we used that multiplication on the left is an automorphism of Cayley graphs).

To simplify the discussion here, let us assume that for each $i\in [m]$ and working modulo $m$, there are elements $h_i,h_i'\in g_iH$ such that $h_i'h_{i+1}$ is an edge in $\mathrm{Cay}_{G}(S)$, and let us further assume all these edges $h_i'h_{i+1}$ are disjoint, i.e. they form a matching. In this scenario, to find a Hamilton cycle in $\mathrm{Cay}_{G}(S)$, it is enough to prove that $\mathrm{Cay}_H(S\cap H)$ is Hamilton connected (meaning there is a Hamilton path between any pair of vertices). So the original problem then boils down to establishing Hamilton connectivity of Cayley graph with an extra `no sparse cut property'.

In general, there is no particular reason to expect the existence of a matching $h_i'h_{i+1}$ that connect the clusters $gH_i$ in a cyclic fashion. Instead, we will need to work with a more complicated connecting structure (see Section~\ref{section:skeleton} for details), that require us to use up to $O(n^{\eps})$ connecting vertices from each left coset rather than just two elements. Thus, the problem reduces to proving the next theorem.

\begin{theorem}\label{thm:maintechnical}There is $n_0\in\mathbb N$ such that for every $n\ge n_0$, $\sigma \ge n^{-1/200}$ and every $k\in[1,2n^{1/200}]$, the following holds. Let $G$ be a $n$-element group and let $S\subset G$ be a (symmetric) generating set with $|S|=\sigma n$ such that $\mathrm{Cay}_G(S)$ has no $\sigma^3/2000$-sparse cuts.
\begin{itemize}
    \item If $\mathrm{Cay}_G(S)$ is non-bipartite, suppose we have pairwise distinct elements $a_1,\ldots,a_k,b_1,\ldots,b_k\in G$. 
    \item If $\mathrm{Cay}_G(S)$ is bipartite with parts $H$ and $\Bar{g}H$ for some subgroup $H\le G$ and  $\Bar{g}\in G$, suppose we have pairwise distinct elements $a_1,\ldots, a_k, b_1,\ldots, b_k\in G$ with the pairs $\{a_i,b_i\}$ fully contained in $H$ or in $\Bar{g}H$, and with an equal number of pairs contained in both parts.
\end{itemize}
Then, there is a spanning linear forest\footnote{ A spanning linear forest in a graph $H$ is a collection of vertex-disjoint paths partitioning the vertices of $H$.} $\mathcal F=\{P_1\ldots, P_k\}$ in $\mathrm{Cay}_G(S)$ such that, for each $i\in [k]$, $P_i$ has endpoints $a_i$ and $b_i$. 
\end{theorem}
\subsection{Well-connected case: Theorem~\ref{thm:maintechnical}}
Here we outline the most technical part of our proof (Theorem~\ref{thm:maintechnical}) for which we will use the absorption method via~\ref{RRS:1}--\ref{RRS:4} as mentioned before. To simplify the discussion here, let us assume $\mathrm{Cay}_G(S)$ is non-bipartite and show how to find a Hamilton path in $\mathrm{Cay}_G(S)$ with no restriction on the endpoints. Recall that in Theorem~\ref{thm:maintechnical} we assume that $\mathrm{Cay}_G(S)$ has no $\sigma^3/2000$-sparse cuts, which implies the following (informal) connectivity property (see Section~\ref{sec:robust expansion} for details).
\begin{enumerate}[label=\upshape{\textbf{P}}]
    \item\label{intro:connect} For any $x,y\in G$ and any `small' $Z\subset G$, there is an $(x,y)$-path in $\mathrm{Cay}_G(S)$ of length $O(\sigma^{-4})$ whose internal vertices avoid $Z$.
\end{enumerate}
We will use property~\ref{intro:connect} for two purposes: a) for constructing the absorbing structure in~\ref{RRS:1}, and b) for linking up the path forest in~\ref{RRS:3}. Let us delve first into the absorbing structure we will use. 

\begin{definition}(Absorber) An \emph{absorber} for a vertex $v$ or a $v$-absorber in a graph $G$ is a subgraph $A\subset G$ such that
\begin{itemize}
    \item $A$ has a Hamiltonian path $P_A$, and
    \item There is another path $P$ in $G$ such that $V(P)=V(A)\cup \{v\}$ and $P$ has the same endpoints as $P_A$.
\end{itemize} \end{definition}
To construct our absorbers, we will use a graph $F_k$ which consists of an odd cycle $v_0x_0x_1\ldots x_ky_0y_ky_{k-1}\ldots y_1$ together with internally vertex-disjoint paths $P_1,\ldots,P_k$, where $P_i$ is an $(x_i,y_i)$-path for each $i\in [k]$. Observe $F_k$ contains two $(x_0,y_0)$-paths, one containing $v_0$ and one avoiding $v_0$. For example, when $k$ is even, the two $(x_0,y_0)$-paths are $x_0x_1P_1y_1y_2P_2x_2x_3P_3\ldots y_kP_kx_ky_0$ and $x_0v_0y_1P_1x_1x_2P_2y_2y_3P_3x_3\ldots x_kP_ky_ky_0$, see Figure~\ref{fig:1} for an illustration.
 
\begin{figure}[h!]
    \centering
    \tikzset{every picture/.style={line width=0.75pt}} 

\begin{tikzpicture}[x=0.75pt,y=0.75pt,yscale=-1,xscale=1, scale=0.9]

\draw    (60,98.21) -- (60,130.54) ;
\draw [shift={(60,130.54)}, rotate = 90] [color={rgb, 255:red, 0; green, 0; blue, 0 }  ][fill={rgb, 255:red, 0; green, 0; blue, 0 }  ][line width=0.75]      (0, 0) circle [x radius= 3.35, y radius= 3.35]   ;
\draw [shift={(60,98.21)}, rotate = 90] [color={rgb, 255:red, 0; green, 0; blue, 0 }  ][fill={rgb, 255:red, 0; green, 0; blue, 0 }  ][line width=0.75]      (0, 0) circle [x radius= 3.35, y radius= 3.35]   ;
\draw    (140,98.21) -- (140,129.9) ;
\draw [shift={(140,129.9)}, rotate = 90] [color={rgb, 255:red, 0; green, 0; blue, 0 }  ][fill={rgb, 255:red, 0; green, 0; blue, 0 }  ][line width=0.75]      (0, 0) circle [x radius= 3.35, y radius= 3.35]   ;
\draw [shift={(140,98.21)}, rotate = 90] [color={rgb, 255:red, 0; green, 0; blue, 0 }  ][fill={rgb, 255:red, 0; green, 0; blue, 0 }  ][line width=0.75]      (0, 0) circle [x radius= 3.35, y radius= 3.35]   ;
\draw    (101,79.19) -- (60,98.21) ;
\draw [shift={(60,98.21)}, rotate = 155.12] [color={rgb, 255:red, 0; green, 0; blue, 0 }  ][fill={rgb, 255:red, 0; green, 0; blue, 0 }  ][line width=0.75]      (0, 0) circle [x radius= 3.35, y radius= 3.35]   ;
\draw [shift={(101,79.19)}, rotate = 155.12] [color={rgb, 255:red, 0; green, 0; blue, 0 }  ][fill={rgb, 255:red, 0; green, 0; blue, 0 }  ][line width=0.75]      (0, 0) circle [x radius= 3.35, y radius= 3.35]   ;
\draw    (101,79.19) -- (140,98.21) ;
\draw [shift={(140,98.21)}, rotate = 25.99] [color={rgb, 255:red, 0; green, 0; blue, 0 }  ][fill={rgb, 255:red, 0; green, 0; blue, 0 }  ][line width=0.75]      (0, 0) circle [x radius= 3.35, y radius= 3.35]   ;
\draw [shift={(101,79.19)}, rotate = 25.99] [color={rgb, 255:red, 0; green, 0; blue, 0 }  ][fill={rgb, 255:red, 0; green, 0; blue, 0 }  ][line width=0.75]      (0, 0) circle [x radius= 3.35, y radius= 3.35]   ;
\draw    (80,160.96) -- (121,160.96) ;
\draw [shift={(121,160.96)}, rotate = 0] [color={rgb, 255:red, 0; green, 0; blue, 0 }  ][fill={rgb, 255:red, 0; green, 0; blue, 0 }  ][line width=0.75]      (0, 0) circle [x radius= 3.35, y radius= 3.35]   ;
\draw [shift={(80,160.96)}, rotate = 0] [color={rgb, 255:red, 0; green, 0; blue, 0 }  ][fill={rgb, 255:red, 0; green, 0; blue, 0 }  ][line width=0.75]      (0, 0) circle [x radius= 3.35, y radius= 3.35]   ;
\draw    (60,98.21) .. controls (62.16,97.28) and (63.71,97.89) .. (64.65,100.05) .. controls (65.59,102.21) and (67.14,102.82) .. (69.3,101.89) .. controls (71.46,100.96) and (73.01,101.57) .. (73.95,103.73) .. controls (74.88,105.9) and (76.42,106.51) .. (78.59,105.58) .. controls (80.75,104.65) and (82.3,105.26) .. (83.24,107.42) .. controls (84.18,109.58) and (85.73,110.19) .. (87.89,109.26) .. controls (90.05,108.33) and (91.6,108.94) .. (92.54,111.1) .. controls (93.48,113.26) and (95.03,113.87) .. (97.19,112.94) .. controls (99.35,112.01) and (100.9,112.62) .. (101.84,114.78) .. controls (102.77,116.95) and (104.31,117.56) .. (106.48,116.63) .. controls (108.64,115.7) and (110.19,116.31) .. (111.13,118.47) .. controls (112.07,120.63) and (113.62,121.24) .. (115.78,120.31) .. controls (117.94,119.38) and (119.49,119.99) .. (120.43,122.15) .. controls (121.37,124.31) and (122.92,124.92) .. (125.08,123.99) .. controls (127.24,123.06) and (128.79,123.67) .. (129.73,125.83) .. controls (130.66,128) and (132.21,128.61) .. (134.38,127.68) .. controls (136.54,126.75) and (138.09,127.36) .. (139.02,129.52) -- (140,129.9) -- (140,129.9) ;
\draw    (140,129.9) -- (121,160.96) ;
\draw [shift={(121,160.96)}, rotate = 121.46] [color={rgb, 255:red, 0; green, 0; blue, 0 }  ][fill={rgb, 255:red, 0; green, 0; blue, 0 }  ][line width=0.75]      (0, 0) circle [x radius= 3.35, y radius= 3.35]   ;
\draw [shift={(140,129.9)}, rotate = 121.46] [color={rgb, 255:red, 0; green, 0; blue, 0 }  ][fill={rgb, 255:red, 0; green, 0; blue, 0 }  ][line width=0.75]      (0, 0) circle [x radius= 3.35, y radius= 3.35]   ;
\draw    (60,130.54) -- (80,160.96) ;
\draw [shift={(80,160.96)}, rotate = 56.68] [color={rgb, 255:red, 0; green, 0; blue, 0 }  ][fill={rgb, 255:red, 0; green, 0; blue, 0 }  ][line width=0.75]      (0, 0) circle [x radius= 3.35, y radius= 3.35]   ;
\draw [shift={(60,130.54)}, rotate = 56.68] [color={rgb, 255:red, 0; green, 0; blue, 0 }  ][fill={rgb, 255:red, 0; green, 0; blue, 0 }  ][line width=0.75]      (0, 0) circle [x radius= 3.35, y radius= 3.35]   ;
\draw    (60,130.54) .. controls (62.24,129.79) and (63.73,130.54) .. (64.47,132.78) .. controls (65.2,135.02) and (66.69,135.77) .. (68.93,135.03) .. controls (71.17,134.29) and (72.66,135.04) .. (73.4,137.28) .. controls (74.14,139.52) and (75.63,140.27) .. (77.87,139.52) .. controls (80.11,138.78) and (81.6,139.53) .. (82.33,141.77) .. controls (83.07,144.01) and (84.56,144.76) .. (86.8,144.02) .. controls (89.04,143.27) and (90.53,144.02) .. (91.27,146.26) .. controls (92.01,148.5) and (93.5,149.25) .. (95.74,148.51) .. controls (97.98,147.77) and (99.47,148.52) .. (100.2,150.76) .. controls (100.94,153) and (102.43,153.75) .. (104.67,153) .. controls (106.91,152.26) and (108.4,153.01) .. (109.14,155.25) .. controls (109.88,157.48) and (111.37,158.23) .. (113.6,157.49) .. controls (115.84,156.75) and (117.33,157.5) .. (118.07,159.74) -- (120.5,160.96) -- (120.5,160.96) ;
\draw    (250,98.21) -- (250,130.54) ;
\draw [shift={(250,130.54)}, rotate = 90] [color={rgb, 255:red, 0; green, 0; blue, 0 }  ][fill={rgb, 255:red, 0; green, 0; blue, 0 }  ][line width=0.75]      (0, 0) circle [x radius= 3.35, y radius= 3.35]   ;
\draw [shift={(250,98.21)}, rotate = 90] [color={rgb, 255:red, 0; green, 0; blue, 0 }  ][fill={rgb, 255:red, 0; green, 0; blue, 0 }  ][line width=0.75]      (0, 0) circle [x radius= 3.35, y radius= 3.35]   ;
\draw    (330,98.21) -- (330,129.9) ;
\draw [shift={(330,129.9)}, rotate = 90] [color={rgb, 255:red, 0; green, 0; blue, 0 }  ][fill={rgb, 255:red, 0; green, 0; blue, 0 }  ][line width=0.75]      (0, 0) circle [x radius= 3.35, y radius= 3.35]   ;
\draw [shift={(330,98.21)}, rotate = 90] [color={rgb, 255:red, 0; green, 0; blue, 0 }  ][fill={rgb, 255:red, 0; green, 0; blue, 0 }  ][line width=0.75]      (0, 0) circle [x radius= 3.35, y radius= 3.35]   ;
\draw  [dash pattern={on 4.5pt off 4.5pt}]  (291,79.19) -- (250,98.21) ;
\draw [shift={(250,98.21)}, rotate = 155.12] [color={rgb, 255:red, 0; green, 0; blue, 0 }  ][fill={rgb, 255:red, 0; green, 0; blue, 0 }  ][line width=0.75]      (0, 0) circle [x radius= 3.35, y radius= 3.35]   ;
\draw [shift={(291,79.19)}, rotate = 155.12] [color={rgb, 255:red, 0; green, 0; blue, 0 }  ][fill={rgb, 255:red, 0; green, 0; blue, 0 }  ][line width=0.75]      (0, 0) circle [x radius= 3.35, y radius= 3.35]   ;
\draw  [dash pattern={on 4.5pt off 4.5pt}]  (291,79.19) -- (330,98.21) ;
\draw [shift={(330,98.21)}, rotate = 25.99] [color={rgb, 255:red, 0; green, 0; blue, 0 }  ][fill={rgb, 255:red, 0; green, 0; blue, 0 }  ][line width=0.75]      (0, 0) circle [x radius= 3.35, y radius= 3.35]   ;
\draw [shift={(291,79.19)}, rotate = 25.99] [color={rgb, 255:red, 0; green, 0; blue, 0 }  ][fill={rgb, 255:red, 0; green, 0; blue, 0 }  ][line width=0.75]      (0, 0) circle [x radius= 3.35, y radius= 3.35]   ;
\draw    (270,160.96) -- (311,160.96) ;
\draw [shift={(311,160.96)}, rotate = 0] [color={rgb, 255:red, 0; green, 0; blue, 0 }  ][fill={rgb, 255:red, 0; green, 0; blue, 0 }  ][line width=0.75]      (0, 0) circle [x radius= 3.35, y radius= 3.35]   ;
\draw [shift={(270,160.96)}, rotate = 0] [color={rgb, 255:red, 0; green, 0; blue, 0 }  ][fill={rgb, 255:red, 0; green, 0; blue, 0 }  ][line width=0.75]      (0, 0) circle [x radius= 3.35, y radius= 3.35]   ;
\draw    (250,98.21) .. controls (252.16,97.28) and (253.71,97.89) .. (254.65,100.05) .. controls (255.59,102.21) and (257.14,102.82) .. (259.3,101.89) .. controls (261.46,100.96) and (263.01,101.57) .. (263.95,103.73) .. controls (264.88,105.9) and (266.42,106.51) .. (268.59,105.58) .. controls (270.75,104.65) and (272.3,105.26) .. (273.24,107.42) .. controls (274.18,109.58) and (275.73,110.19) .. (277.89,109.26) .. controls (280.05,108.33) and (281.6,108.94) .. (282.54,111.1) .. controls (283.48,113.26) and (285.03,113.87) .. (287.19,112.94) .. controls (289.35,112.01) and (290.9,112.62) .. (291.84,114.78) .. controls (292.77,116.95) and (294.31,117.56) .. (296.48,116.63) .. controls (298.64,115.7) and (300.19,116.31) .. (301.13,118.47) .. controls (302.07,120.63) and (303.62,121.24) .. (305.78,120.31) .. controls (307.94,119.38) and (309.49,119.99) .. (310.43,122.15) .. controls (311.37,124.31) and (312.92,124.92) .. (315.08,123.99) .. controls (317.24,123.06) and (318.79,123.67) .. (319.73,125.83) .. controls (320.66,128) and (322.21,128.61) .. (324.38,127.68) .. controls (326.54,126.75) and (328.09,127.36) .. (329.02,129.52) -- (330,129.9) -- (330,129.9) ;
\draw  [dash pattern={on 4.5pt off 4.5pt}]  (330,129.9) -- (311,160.96) ;
\draw [shift={(311,160.96)}, rotate = 121.46] [color={rgb, 255:red, 0; green, 0; blue, 0 }  ][fill={rgb, 255:red, 0; green, 0; blue, 0 }  ][line width=0.75]      (0, 0) circle [x radius= 3.35, y radius= 3.35]   ;
\draw [shift={(330,129.9)}, rotate = 121.46] [color={rgb, 255:red, 0; green, 0; blue, 0 }  ][fill={rgb, 255:red, 0; green, 0; blue, 0 }  ][line width=0.75]      (0, 0) circle [x radius= 3.35, y radius= 3.35]   ;
\draw  [dash pattern={on 4.5pt off 4.5pt}]  (250,130.54) -- (270,160.96) ;
\draw [shift={(270,160.96)}, rotate = 56.68] [color={rgb, 255:red, 0; green, 0; blue, 0 }  ][fill={rgb, 255:red, 0; green, 0; blue, 0 }  ][line width=0.75]      (0, 0) circle [x radius= 3.35, y radius= 3.35]   ;
\draw [shift={(250,130.54)}, rotate = 56.68] [color={rgb, 255:red, 0; green, 0; blue, 0 }  ][fill={rgb, 255:red, 0; green, 0; blue, 0 }  ][line width=0.75]      (0, 0) circle [x radius= 3.35, y radius= 3.35]   ;
\draw    (249,129.9) .. controls (251.23,129.16) and (252.72,129.91) .. (253.47,132.14) .. controls (254.21,134.38) and (255.7,135.13) .. (257.94,134.38) .. controls (260.18,133.63) and (261.67,134.38) .. (262.41,136.62) .. controls (263.15,138.86) and (264.64,139.61) .. (266.88,138.86) .. controls (269.12,138.11) and (270.61,138.86) .. (271.35,141.1) .. controls (272.09,143.34) and (273.58,144.09) .. (275.82,143.34) .. controls (278.06,142.59) and (279.55,143.34) .. (280.29,145.58) .. controls (281.03,147.82) and (282.52,148.57) .. (284.76,147.82) .. controls (287,147.07) and (288.49,147.82) .. (289.23,150.06) .. controls (289.97,152.3) and (291.46,153.05) .. (293.7,152.3) .. controls (295.94,151.55) and (297.43,152.3) .. (298.17,154.54) .. controls (298.91,156.78) and (300.4,157.53) .. (302.64,156.78) .. controls (304.88,156.03) and (306.37,156.78) .. (307.12,159.02) -- (311,160.96) -- (311,160.96) ;
\draw  [dash pattern={on 4.5pt off 4.5pt}]  (440,98.21) -- (440,130.54) ;
\draw [shift={(440,130.54)}, rotate = 90] [color={rgb, 255:red, 0; green, 0; blue, 0 }  ][fill={rgb, 255:red, 0; green, 0; blue, 0 }  ][line width=0.75]      (0, 0) circle [x radius= 3.35, y radius= 3.35]   ;
\draw [shift={(440,98.21)}, rotate = 90] [color={rgb, 255:red, 0; green, 0; blue, 0 }  ][fill={rgb, 255:red, 0; green, 0; blue, 0 }  ][line width=0.75]      (0, 0) circle [x radius= 3.35, y radius= 3.35]   ;
\draw  [dash pattern={on 4.5pt off 4.5pt}]  (520,98.21) -- (520,129.9) ;
\draw [shift={(520,129.9)}, rotate = 90] [color={rgb, 255:red, 0; green, 0; blue, 0 }  ][fill={rgb, 255:red, 0; green, 0; blue, 0 }  ][line width=0.75]      (0, 0) circle [x radius= 3.35, y radius= 3.35]   ;
\draw [shift={(520,98.21)}, rotate = 90] [color={rgb, 255:red, 0; green, 0; blue, 0 }  ][fill={rgb, 255:red, 0; green, 0; blue, 0 }  ][line width=0.75]      (0, 0) circle [x radius= 3.35, y radius= 3.35]   ;
\draw    (481,79.19) -- (440,98.21) ;
\draw [shift={(440,98.21)}, rotate = 155.12] [color={rgb, 255:red, 0; green, 0; blue, 0 }  ][fill={rgb, 255:red, 0; green, 0; blue, 0 }  ][line width=0.75]      (0, 0) circle [x radius= 3.35, y radius= 3.35]   ;
\draw [shift={(481,79.19)}, rotate = 155.12] [color={rgb, 255:red, 0; green, 0; blue, 0 }  ][fill={rgb, 255:red, 0; green, 0; blue, 0 }  ][line width=0.75]      (0, 0) circle [x radius= 3.35, y radius= 3.35]   ;
\draw    (481,79.19) -- (520,98.21) ;
\draw [shift={(520,98.21)}, rotate = 25.99] [color={rgb, 255:red, 0; green, 0; blue, 0 }  ][fill={rgb, 255:red, 0; green, 0; blue, 0 }  ][line width=0.75]      (0, 0) circle [x radius= 3.35, y radius= 3.35]   ;
\draw [shift={(481,79.19)}, rotate = 25.99] [color={rgb, 255:red, 0; green, 0; blue, 0 }  ][fill={rgb, 255:red, 0; green, 0; blue, 0 }  ][line width=0.75]      (0, 0) circle [x radius= 3.35, y radius= 3.35]   ;
\draw  [dash pattern={on 4.5pt off 4.5pt}]  (460,160.96) -- (501,160.96) ;
\draw [shift={(501,160.96)}, rotate = 0] [color={rgb, 255:red, 0; green, 0; blue, 0 }  ][fill={rgb, 255:red, 0; green, 0; blue, 0 }  ][line width=0.75]      (0, 0) circle [x radius= 3.35, y radius= 3.35]   ;
\draw [shift={(460,160.96)}, rotate = 0] [color={rgb, 255:red, 0; green, 0; blue, 0 }  ][fill={rgb, 255:red, 0; green, 0; blue, 0 }  ][line width=0.75]      (0, 0) circle [x radius= 3.35, y radius= 3.35]   ;
\draw    (520,129.9) -- (501,160.96) ;
\draw [shift={(501,160.96)}, rotate = 121.46] [color={rgb, 255:red, 0; green, 0; blue, 0 }  ][fill={rgb, 255:red, 0; green, 0; blue, 0 }  ][line width=0.75]      (0, 0) circle [x radius= 3.35, y radius= 3.35]   ;
\draw [shift={(520,129.9)}, rotate = 121.46] [color={rgb, 255:red, 0; green, 0; blue, 0 }  ][fill={rgb, 255:red, 0; green, 0; blue, 0 }  ][line width=0.75]      (0, 0) circle [x radius= 3.35, y radius= 3.35]   ;
\draw    (440,130.54) -- (460,160.96) ;
\draw [shift={(460,160.96)}, rotate = 56.68] [color={rgb, 255:red, 0; green, 0; blue, 0 }  ][fill={rgb, 255:red, 0; green, 0; blue, 0 }  ][line width=0.75]      (0, 0) circle [x radius= 3.35, y radius= 3.35]   ;
\draw [shift={(440,130.54)}, rotate = 56.68] [color={rgb, 255:red, 0; green, 0; blue, 0 }  ][fill={rgb, 255:red, 0; green, 0; blue, 0 }  ][line width=0.75]      (0, 0) circle [x radius= 3.35, y radius= 3.35]   ;
\draw    (439,129.9) .. controls (441.23,129.16) and (442.72,129.91) .. (443.47,132.14) .. controls (444.21,134.38) and (445.7,135.13) .. (447.94,134.38) .. controls (450.18,133.63) and (451.67,134.38) .. (452.41,136.62) .. controls (453.15,138.86) and (454.64,139.61) .. (456.88,138.86) .. controls (459.12,138.11) and (460.61,138.86) .. (461.35,141.1) .. controls (462.09,143.34) and (463.58,144.09) .. (465.82,143.34) .. controls (468.06,142.59) and (469.55,143.34) .. (470.29,145.58) .. controls (471.03,147.82) and (472.52,148.57) .. (474.76,147.82) .. controls (477,147.07) and (478.49,147.82) .. (479.23,150.06) .. controls (479.97,152.3) and (481.46,153.05) .. (483.7,152.3) .. controls (485.94,151.55) and (487.43,152.3) .. (488.17,154.54) .. controls (488.91,156.78) and (490.4,157.53) .. (492.64,156.78) .. controls (494.88,156.03) and (496.37,156.78) .. (497.12,159.02) -- (501,160.96) -- (501,160.96) ;
\draw    (440,98.21) .. controls (442.16,97.28) and (443.71,97.89) .. (444.65,100.05) .. controls (445.59,102.21) and (447.14,102.82) .. (449.3,101.89) .. controls (451.46,100.96) and (453.01,101.57) .. (453.95,103.73) .. controls (454.88,105.9) and (456.42,106.51) .. (458.59,105.58) .. controls (460.75,104.65) and (462.3,105.26) .. (463.24,107.42) .. controls (464.18,109.58) and (465.73,110.19) .. (467.89,109.26) .. controls (470.05,108.33) and (471.6,108.94) .. (472.54,111.1) .. controls (473.48,113.26) and (475.03,113.87) .. (477.19,112.94) .. controls (479.35,112.01) and (480.9,112.62) .. (481.84,114.78) .. controls (482.77,116.95) and (484.31,117.56) .. (486.48,116.63) .. controls (488.64,115.7) and (490.19,116.31) .. (491.13,118.47) .. controls (492.07,120.63) and (493.62,121.24) .. (495.78,120.31) .. controls (497.94,119.38) and (499.49,119.99) .. (500.43,122.15) .. controls (501.37,124.31) and (502.92,124.92) .. (505.08,123.99) .. controls (507.24,123.06) and (508.79,123.67) .. (509.73,125.83) .. controls (510.66,128) and (512.21,128.61) .. (514.38,127.68) .. controls (516.54,126.75) and (518.09,127.36) .. (519.02,129.52) -- (520,129.9) -- (520,129.9) ;

\draw (96,59.25) node [anchor=north west][inner sep=0.75pt]    {$v$};
\draw (142,97.22) node [anchor=north west][inner sep=0.75pt]    {$x_{0}$};
\draw (63,159.34) node [anchor=north west][inner sep=0.75pt]    {$y_{0}$};
\draw (142,128.91) node [anchor=north west][inner sep=0.75pt]    {$x_{1}$};
\draw (40,121.94) node [anchor=north west][inner sep=0.75pt]    {$y_{2}$};
\draw (123,159.97) node [anchor=north west][inner sep=0.75pt]    {$x_{2}$};
\draw (40,82.64) node [anchor=north west][inner sep=0.75pt]    {$y_{1}$};
\draw (286,59.62) node [anchor=north west][inner sep=0.75pt]    {$v$};
\draw (332,97.22) node [anchor=north west][inner sep=0.75pt]    {$x_{0}$};
\draw (253,159.34) node [anchor=north west][inner sep=0.75pt]    {$y_{0}$};
\draw (332,128.91) node [anchor=north west][inner sep=0.75pt]    {$x_{1}$};
\draw (230,121.94) node [anchor=north west][inner sep=0.75pt]    {$y_{2}$};
\draw (313,159.97) node [anchor=north west][inner sep=0.75pt]    {$x_{2}$};
\draw (230,82.64) node [anchor=north west][inner sep=0.75pt]    {$y_{1}$};
\draw (476,59.18) node [anchor=north west][inner sep=0.75pt]    {$v$};
\draw (522,97.22) node [anchor=north west][inner sep=0.75pt]    {$x_{0}$};
\draw (443,159.34) node [anchor=north west][inner sep=0.75pt]    {$y_{0}$};
\draw (522,128.91) node [anchor=north west][inner sep=0.75pt]    {$x_{1}$};
\draw (420,121.94) node [anchor=north west][inner sep=0.75pt]    {$y_{2}$};
\draw (503,159.97) node [anchor=north west][inner sep=0.75pt]    {$x_{2}$};
\draw (420,82.64) node [anchor=north west][inner sep=0.75pt]    {$y_{1}$};

\end{tikzpicture}
    \caption{A picture of the $v$-absorber $F_5$ and the corresponding $(x_0,y_0)$-paths.}
    \label{fig:1}
\end{figure}
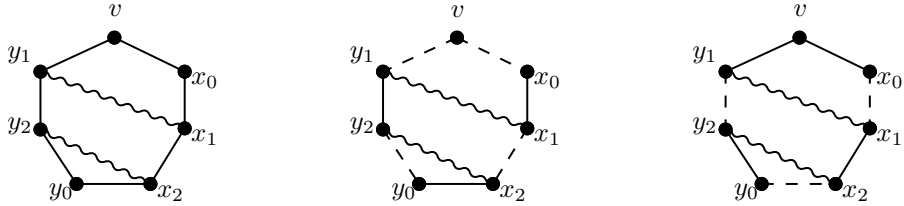
To construct such an absorber, our strategy is to start with an odd cycle and then use \ref{intro:connect} repeatedly to form the paths $P_1,\ldots, P_k$. There are some problems with this approach, though. Firstly, we need to be able to use~\ref{intro:connect} until we construct enough absorbers so that they can deal with the leftover after~\ref{RRS:2}. Moreover,~\ref{intro:connect} has a restriction on how many vertices it can avoid, and so it imposes a restriction on how many absorbers we can construct with this approach. Secondly, the more we rely on property~\ref{intro:connect} to connect vertices, the more we might deteriorate the regularity properties of $\mathrm{Cay}_G(S)$. This is indeed the main problem with this approach, as the number of paths we can find in~\ref{RRS:2} strongly depends on how regular the graph is.  

To find a path forest in $\mathrm{Cay}_G(S)$ in~\ref{RRS:2}, we use results around the Magnant--Martin conjecture~\cite{magnantmartin} about the number of paths needed to partition a $d$-regular graph (conjecturally, $n/d$ paths should be sufficient in any $d$-regular graph). However, after finding our absorbers, the resulting graph might no longer be regular, so we need to tailor these results to work in the approximately regular case as well. This can be easily achieved by adding a few extra vertices and then applying these results for regular graphs, from which one gets that any $n$-vertex $d\pm d'$ regular graph can be partitioned into at most $O(\frac{n\log d}{d}+\frac{nd'}{d})$ paths (see Lemma~\ref{lem:directed:magnant}). This implies that the larger $d'$ is, the more paths we need to cover the graph, which increases the number of absorbers we need to successfully stitch the paths together. Therefore, the crux of this approach is to construct sufficiently many absorbers without significantly affecting the degree of the leftover graph. 

To make this argument work, we will exploit the translation invariance of $\mathrm{Cay}_G(S)$ to construct a random-like absorber. We start with the following observation: for a set $F\subset G$ and an element $g\in G$, the way $\mathrm{Cay}_G(S)$ is defined implies that the graphs induced by $F$ and $gF$ are isomorphic. Thus, if $A$ is an absorber for some $x\in G$ with endpoints $x_0$ and $y_0$, then $gA$ is an absorber for $gx$ with endpoint $gx_0$ and $gy_0$. Using this observation and having constructed a $v$-absorber $A$ with endpoints $x_0$ and $y_0$, we let $g_1,\ldots, g_m$ be elements of $G$ chosen uniformly at random with repetition. Then, for $F=V(A)$, the sets $R_{x_0}=\{g_ix_0:i\in [m]\}$, $R_{y_0}=\{g_iy_0:i\in [m]\}$, $R_v=\{g_iv:i\in [m]\}$ and $R_F=\cup_{i\in [m]} g_iF$ are random-like in the sense that with high probability every element $x\in G$ satisfies $d(x,R_{F})\approx \sigma m|F|$ and $d(x,R_{a})\approx \sigma m$ for $a\in\{v,x_0,y_0\}$ (see Lemma~\ref{lem:randomtranslates}). As there are few overlapping translates $g_iF$ and $g_jF$, $R_F$ is essentially a disjoint union of translates of $A$. Moreover, one can prove that \ref{intro:connect} is robust enough so that the connecting paths have all their internal vertices in $R_v$ (see Lemma~\ref{lem:connecting:translates}). This is~\ref{RRS:1}.

For~\ref{RRS:2} we will use results around the Magnant--Martin conjecture~\cite{christoph2025cycle}. As previously mentioned, in the approximately regular case, the number of paths will be determined by the degree of irregularity, which in our case depends on how the degree deviates from $\sigma m|F|$. After~\ref{RRS:2}, we link up the endpoints of the path-forest and also the absorbers in $R_F$ using property~\ref{intro:connect} through the set $R_v$. To perform this efficiently, we first link up the absorbers using a directed version of the Magnant--Martin conjecture in an auxiliary digraph $D$ defined as follows: the vertex set is $V(D)=[m]$ and we put an edge from $i$ to $j$ whenever $g_iy_0$ is adjacent to $g_jx_0$ in $\mathrm{Cay}_G(S)$. As every vertex has roughly $\sigma m$ neighbours in both $R_{x_0}$ and $R_{y_0}$, the auxiliary digraph $D$ is approximately regular as well and thus we can partition it using few directed paths. Each directed path in $D$ corresponds to a bunch of absorbers joined together so that it can now absorb many vertices at the same time (see Figure~\ref{fig:2} below). These are~\ref{RRS:2} and~\ref{RRS:3}.
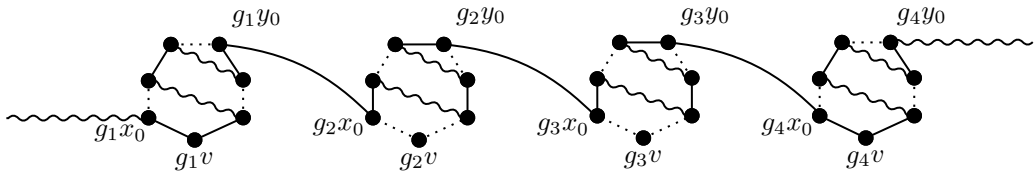
\begin{figure}[h!]
    \centering

\tikzset{every picture/.style={line width=0.75pt}} 

\begin{tikzpicture}[x=0.75pt,y=0.75pt,yscale=-1,xscale=1]

\draw  [dash pattern={on 0.84pt off 2.51pt}]  (174.19,108.27) -- (174.17,89.26) ;
\draw [shift={(174.17,89.26)}, rotate = 269.92] [color={rgb, 255:red, 0; green, 0; blue, 0 }  ][fill={rgb, 255:red, 0; green, 0; blue, 0 }  ][line width=0.75]      (0, 0) circle [x radius= 3.35, y radius= 3.35]   ;
\draw [shift={(174.19,108.27)}, rotate = 269.92] [color={rgb, 255:red, 0; green, 0; blue, 0 }  ][fill={rgb, 255:red, 0; green, 0; blue, 0 }  ][line width=0.75]      (0, 0) circle [x radius= 3.35, y radius= 3.35]   ;
\draw  [dash pattern={on 0.84pt off 2.51pt}]  (126.98,108.33) -- (126.96,89.7) ;
\draw [shift={(126.96,89.7)}, rotate = 269.92] [color={rgb, 255:red, 0; green, 0; blue, 0 }  ][fill={rgb, 255:red, 0; green, 0; blue, 0 }  ][line width=0.75]      (0, 0) circle [x radius= 3.35, y radius= 3.35]   ;
\draw [shift={(126.98,108.33)}, rotate = 269.92] [color={rgb, 255:red, 0; green, 0; blue, 0 }  ][fill={rgb, 255:red, 0; green, 0; blue, 0 }  ][line width=0.75]      (0, 0) circle [x radius= 3.35, y radius= 3.35]   ;
\draw    (150.01,119.48) -- (174.19,108.27) ;
\draw [shift={(174.19,108.27)}, rotate = 335.12] [color={rgb, 255:red, 0; green, 0; blue, 0 }  ][fill={rgb, 255:red, 0; green, 0; blue, 0 }  ][line width=0.75]      (0, 0) circle [x radius= 3.35, y radius= 3.35]   ;
\draw [shift={(150.01,119.48)}, rotate = 335.12] [color={rgb, 255:red, 0; green, 0; blue, 0 }  ][fill={rgb, 255:red, 0; green, 0; blue, 0 }  ][line width=0.75]      (0, 0) circle [x radius= 3.35, y radius= 3.35]   ;
\draw    (150.01,119.48) -- (126.98,108.33) ;
\draw [shift={(126.98,108.33)}, rotate = 205.83] [color={rgb, 255:red, 0; green, 0; blue, 0 }  ][fill={rgb, 255:red, 0; green, 0; blue, 0 }  ][line width=0.75]      (0, 0) circle [x radius= 3.35, y radius= 3.35]   ;
\draw [shift={(150.01,119.48)}, rotate = 205.83] [color={rgb, 255:red, 0; green, 0; blue, 0 }  ][fill={rgb, 255:red, 0; green, 0; blue, 0 }  ][line width=0.75]      (0, 0) circle [x radius= 3.35, y radius= 3.35]   ;
\draw  [dash pattern={on 0.84pt off 2.51pt}]  (162.34,71.39) -- (138.14,71.42) ;
\draw [shift={(138.14,71.42)}, rotate = 179.92] [color={rgb, 255:red, 0; green, 0; blue, 0 }  ][fill={rgb, 255:red, 0; green, 0; blue, 0 }  ][line width=0.75]      (0, 0) circle [x radius= 3.35, y radius= 3.35]   ;
\draw [shift={(162.34,71.39)}, rotate = 179.92] [color={rgb, 255:red, 0; green, 0; blue, 0 }  ][fill={rgb, 255:red, 0; green, 0; blue, 0 }  ][line width=0.75]      (0, 0) circle [x radius= 3.35, y radius= 3.35]   ;
\draw    (174.19,108.27) .. controls (172.03,109.21) and (170.48,108.6) .. (169.54,106.44) .. controls (168.6,104.28) and (167.05,103.67) .. (164.89,104.61) .. controls (162.73,105.55) and (161.18,104.94) .. (160.23,102.78) .. controls (159.29,100.62) and (157.74,100.01) .. (155.58,100.95) .. controls (153.42,101.89) and (151.87,101.28) .. (150.93,99.12) .. controls (149.98,96.96) and (148.43,96.35) .. (146.27,97.29) .. controls (144.11,98.23) and (142.56,97.62) .. (141.62,95.46) .. controls (140.68,93.3) and (139.13,92.69) .. (136.97,93.63) .. controls (134.81,94.58) and (133.26,93.97) .. (132.31,91.81) .. controls (131.37,89.65) and (129.82,89.04) .. (127.66,89.98) -- (126.96,89.7) -- (126.96,89.7) ;
\draw    (126.96,89.7) -- (138.14,71.42) ;
\draw [shift={(138.14,71.42)}, rotate = 301.47] [color={rgb, 255:red, 0; green, 0; blue, 0 }  ][fill={rgb, 255:red, 0; green, 0; blue, 0 }  ][line width=0.75]      (0, 0) circle [x radius= 3.35, y radius= 3.35]   ;
\draw [shift={(126.96,89.7)}, rotate = 301.47] [color={rgb, 255:red, 0; green, 0; blue, 0 }  ][fill={rgb, 255:red, 0; green, 0; blue, 0 }  ][line width=0.75]      (0, 0) circle [x radius= 3.35, y radius= 3.35]   ;
\draw    (174.17,89.26) -- (162.34,71.39) ;
\draw [shift={(162.34,71.39)}, rotate = 236.5] [color={rgb, 255:red, 0; green, 0; blue, 0 }  ][fill={rgb, 255:red, 0; green, 0; blue, 0 }  ][line width=0.75]      (0, 0) circle [x radius= 3.35, y radius= 3.35]   ;
\draw [shift={(174.17,89.26)}, rotate = 236.5] [color={rgb, 255:red, 0; green, 0; blue, 0 }  ][fill={rgb, 255:red, 0; green, 0; blue, 0 }  ][line width=0.75]      (0, 0) circle [x radius= 3.35, y radius= 3.35]   ;
\draw    (174.17,89.26) .. controls (171.93,90.01) and (170.44,89.26) .. (169.69,87.03) .. controls (168.95,84.79) and (167.46,84.04) .. (165.22,84.79) .. controls (162.99,85.54) and (161.5,84.79) .. (160.75,82.56) .. controls (160,80.33) and (158.5,79.58) .. (156.27,80.33) .. controls (154.03,81.08) and (152.54,80.33) .. (151.8,78.09) .. controls (151.05,75.86) and (149.56,75.11) .. (147.33,75.86) .. controls (145.1,76.61) and (143.6,75.86) .. (142.85,73.63) -- (138.44,71.42) -- (138.44,71.42) ;
\draw    (286.19,108.27) -- (286.17,89.26) ;
\draw [shift={(286.17,89.26)}, rotate = 269.92] [color={rgb, 255:red, 0; green, 0; blue, 0 }  ][fill={rgb, 255:red, 0; green, 0; blue, 0 }  ][line width=0.75]      (0, 0) circle [x radius= 3.35, y radius= 3.35]   ;
\draw [shift={(286.19,108.27)}, rotate = 269.92] [color={rgb, 255:red, 0; green, 0; blue, 0 }  ][fill={rgb, 255:red, 0; green, 0; blue, 0 }  ][line width=0.75]      (0, 0) circle [x radius= 3.35, y radius= 3.35]   ;
\draw    (238.98,108.33) -- (238.96,89.7) ;
\draw [shift={(238.96,89.7)}, rotate = 269.92] [color={rgb, 255:red, 0; green, 0; blue, 0 }  ][fill={rgb, 255:red, 0; green, 0; blue, 0 }  ][line width=0.75]      (0, 0) circle [x radius= 3.35, y radius= 3.35]   ;
\draw [shift={(238.98,108.33)}, rotate = 269.92] [color={rgb, 255:red, 0; green, 0; blue, 0 }  ][fill={rgb, 255:red, 0; green, 0; blue, 0 }  ][line width=0.75]      (0, 0) circle [x radius= 3.35, y radius= 3.35]   ;
\draw  [dash pattern={on 0.84pt off 2.51pt}]  (262.01,119.48) -- (286.19,108.27) ;
\draw [shift={(286.19,108.27)}, rotate = 335.12] [color={rgb, 255:red, 0; green, 0; blue, 0 }  ][fill={rgb, 255:red, 0; green, 0; blue, 0 }  ][line width=0.75]      (0, 0) circle [x radius= 3.35, y radius= 3.35]   ;
\draw [shift={(262.01,119.48)}, rotate = 335.12] [color={rgb, 255:red, 0; green, 0; blue, 0 }  ][fill={rgb, 255:red, 0; green, 0; blue, 0 }  ][line width=0.75]      (0, 0) circle [x radius= 3.35, y radius= 3.35]   ;
\draw  [dash pattern={on 0.84pt off 2.51pt}]  (262.01,119.48) -- (238.98,108.33) ;
\draw [shift={(238.98,108.33)}, rotate = 205.83] [color={rgb, 255:red, 0; green, 0; blue, 0 }  ][fill={rgb, 255:red, 0; green, 0; blue, 0 }  ][line width=0.75]      (0, 0) circle [x radius= 3.35, y radius= 3.35]   ;
\draw [shift={(262.01,119.48)}, rotate = 205.83] [color={rgb, 255:red, 0; green, 0; blue, 0 }  ][fill={rgb, 255:red, 0; green, 0; blue, 0 }  ][line width=0.75]      (0, 0) circle [x radius= 3.35, y radius= 3.35]   ;
\draw    (274.34,71.39) -- (250.14,71.42) ;
\draw [shift={(250.14,71.42)}, rotate = 179.92] [color={rgb, 255:red, 0; green, 0; blue, 0 }  ][fill={rgb, 255:red, 0; green, 0; blue, 0 }  ][line width=0.75]      (0, 0) circle [x radius= 3.35, y radius= 3.35]   ;
\draw [shift={(274.34,71.39)}, rotate = 179.92] [color={rgb, 255:red, 0; green, 0; blue, 0 }  ][fill={rgb, 255:red, 0; green, 0; blue, 0 }  ][line width=0.75]      (0, 0) circle [x radius= 3.35, y radius= 3.35]   ;
\draw    (286.19,108.27) .. controls (284.03,109.21) and (282.48,108.6) .. (281.54,106.44) .. controls (280.6,104.28) and (279.05,103.67) .. (276.89,104.61) .. controls (274.73,105.55) and (273.18,104.94) .. (272.23,102.78) .. controls (271.29,100.62) and (269.74,100.01) .. (267.58,100.95) .. controls (265.42,101.89) and (263.87,101.28) .. (262.93,99.12) .. controls (261.98,96.96) and (260.43,96.35) .. (258.27,97.29) .. controls (256.11,98.23) and (254.56,97.62) .. (253.62,95.46) .. controls (252.68,93.3) and (251.13,92.69) .. (248.97,93.63) .. controls (246.81,94.58) and (245.26,93.97) .. (244.31,91.81) .. controls (243.37,89.65) and (241.82,89.04) .. (239.66,89.98) -- (238.96,89.7) -- (238.96,89.7) ;
\draw  [dash pattern={on 0.84pt off 2.51pt}]  (238.96,89.7) -- (250.14,71.42) ;
\draw [shift={(250.14,71.42)}, rotate = 301.47] [color={rgb, 255:red, 0; green, 0; blue, 0 }  ][fill={rgb, 255:red, 0; green, 0; blue, 0 }  ][line width=0.75]      (0, 0) circle [x radius= 3.35, y radius= 3.35]   ;
\draw [shift={(238.96,89.7)}, rotate = 301.47] [color={rgb, 255:red, 0; green, 0; blue, 0 }  ][fill={rgb, 255:red, 0; green, 0; blue, 0 }  ][line width=0.75]      (0, 0) circle [x radius= 3.35, y radius= 3.35]   ;
\draw  [dash pattern={on 0.84pt off 2.51pt}]  (286.17,89.26) -- (274.34,71.39) ;
\draw [shift={(274.34,71.39)}, rotate = 236.5] [color={rgb, 255:red, 0; green, 0; blue, 0 }  ][fill={rgb, 255:red, 0; green, 0; blue, 0 }  ][line width=0.75]      (0, 0) circle [x radius= 3.35, y radius= 3.35]   ;
\draw [shift={(286.17,89.26)}, rotate = 236.5] [color={rgb, 255:red, 0; green, 0; blue, 0 }  ][fill={rgb, 255:red, 0; green, 0; blue, 0 }  ][line width=0.75]      (0, 0) circle [x radius= 3.35, y radius= 3.35]   ;
\draw    (286.17,89.26) .. controls (283.93,90.01) and (282.44,89.26) .. (281.69,87.03) .. controls (280.95,84.79) and (279.46,84.04) .. (277.22,84.79) .. controls (274.99,85.54) and (273.5,84.79) .. (272.75,82.56) .. controls (272,80.33) and (270.5,79.58) .. (268.27,80.33) .. controls (266.03,81.08) and (264.54,80.33) .. (263.8,78.09) .. controls (263.05,75.86) and (261.56,75.11) .. (259.33,75.86) .. controls (257.1,76.61) and (255.6,75.86) .. (254.85,73.63) -- (250.44,71.42) -- (250.44,71.42) ;
\draw    (398.19,107.27) -- (398.17,88.26) ;
\draw [shift={(398.17,88.26)}, rotate = 269.92] [color={rgb, 255:red, 0; green, 0; blue, 0 }  ][fill={rgb, 255:red, 0; green, 0; blue, 0 }  ][line width=0.75]      (0, 0) circle [x radius= 3.35, y radius= 3.35]   ;
\draw [shift={(398.19,107.27)}, rotate = 269.92] [color={rgb, 255:red, 0; green, 0; blue, 0 }  ][fill={rgb, 255:red, 0; green, 0; blue, 0 }  ][line width=0.75]      (0, 0) circle [x radius= 3.35, y radius= 3.35]   ;
\draw    (350.98,107.33) -- (350.96,88.7) ;
\draw [shift={(350.96,88.7)}, rotate = 269.92] [color={rgb, 255:red, 0; green, 0; blue, 0 }  ][fill={rgb, 255:red, 0; green, 0; blue, 0 }  ][line width=0.75]      (0, 0) circle [x radius= 3.35, y radius= 3.35]   ;
\draw [shift={(350.98,107.33)}, rotate = 269.92] [color={rgb, 255:red, 0; green, 0; blue, 0 }  ][fill={rgb, 255:red, 0; green, 0; blue, 0 }  ][line width=0.75]      (0, 0) circle [x radius= 3.35, y radius= 3.35]   ;
\draw  [dash pattern={on 0.84pt off 2.51pt}]  (374.01,118.48) -- (398.19,107.27) ;
\draw [shift={(398.19,107.27)}, rotate = 335.12] [color={rgb, 255:red, 0; green, 0; blue, 0 }  ][fill={rgb, 255:red, 0; green, 0; blue, 0 }  ][line width=0.75]      (0, 0) circle [x radius= 3.35, y radius= 3.35]   ;
\draw [shift={(374.01,118.48)}, rotate = 335.12] [color={rgb, 255:red, 0; green, 0; blue, 0 }  ][fill={rgb, 255:red, 0; green, 0; blue, 0 }  ][line width=0.75]      (0, 0) circle [x radius= 3.35, y radius= 3.35]   ;
\draw  [dash pattern={on 0.84pt off 2.51pt}]  (374.01,118.48) -- (350.98,107.33) ;
\draw [shift={(350.98,107.33)}, rotate = 205.83] [color={rgb, 255:red, 0; green, 0; blue, 0 }  ][fill={rgb, 255:red, 0; green, 0; blue, 0 }  ][line width=0.75]      (0, 0) circle [x radius= 3.35, y radius= 3.35]   ;
\draw [shift={(374.01,118.48)}, rotate = 205.83] [color={rgb, 255:red, 0; green, 0; blue, 0 }  ][fill={rgb, 255:red, 0; green, 0; blue, 0 }  ][line width=0.75]      (0, 0) circle [x radius= 3.35, y radius= 3.35]   ;
\draw    (386.34,70.39) -- (362.14,70.42) ;
\draw [shift={(362.14,70.42)}, rotate = 179.92] [color={rgb, 255:red, 0; green, 0; blue, 0 }  ][fill={rgb, 255:red, 0; green, 0; blue, 0 }  ][line width=0.75]      (0, 0) circle [x radius= 3.35, y radius= 3.35]   ;
\draw [shift={(386.34,70.39)}, rotate = 179.92] [color={rgb, 255:red, 0; green, 0; blue, 0 }  ][fill={rgb, 255:red, 0; green, 0; blue, 0 }  ][line width=0.75]      (0, 0) circle [x radius= 3.35, y radius= 3.35]   ;
\draw    (398.19,107.27) .. controls (396.03,108.21) and (394.48,107.6) .. (393.54,105.44) .. controls (392.6,103.28) and (391.05,102.67) .. (388.89,103.61) .. controls (386.73,104.55) and (385.18,103.94) .. (384.23,101.78) .. controls (383.29,99.62) and (381.74,99.01) .. (379.58,99.95) .. controls (377.42,100.89) and (375.87,100.28) .. (374.93,98.12) .. controls (373.98,95.96) and (372.43,95.35) .. (370.27,96.29) .. controls (368.11,97.23) and (366.56,96.62) .. (365.62,94.46) .. controls (364.68,92.3) and (363.13,91.69) .. (360.97,92.63) .. controls (358.81,93.58) and (357.26,92.97) .. (356.31,90.81) .. controls (355.37,88.65) and (353.82,88.04) .. (351.66,88.98) -- (350.96,88.7) -- (350.96,88.7) ;
\draw  [dash pattern={on 0.84pt off 2.51pt}]  (350.96,88.7) -- (362.14,70.42) ;
\draw [shift={(362.14,70.42)}, rotate = 301.47] [color={rgb, 255:red, 0; green, 0; blue, 0 }  ][fill={rgb, 255:red, 0; green, 0; blue, 0 }  ][line width=0.75]      (0, 0) circle [x radius= 3.35, y radius= 3.35]   ;
\draw [shift={(350.96,88.7)}, rotate = 301.47] [color={rgb, 255:red, 0; green, 0; blue, 0 }  ][fill={rgb, 255:red, 0; green, 0; blue, 0 }  ][line width=0.75]      (0, 0) circle [x radius= 3.35, y radius= 3.35]   ;
\draw  [dash pattern={on 0.84pt off 2.51pt}]  (398.17,88.26) -- (386.34,70.39) ;
\draw [shift={(386.34,70.39)}, rotate = 236.5] [color={rgb, 255:red, 0; green, 0; blue, 0 }  ][fill={rgb, 255:red, 0; green, 0; blue, 0 }  ][line width=0.75]      (0, 0) circle [x radius= 3.35, y radius= 3.35]   ;
\draw [shift={(398.17,88.26)}, rotate = 236.5] [color={rgb, 255:red, 0; green, 0; blue, 0 }  ][fill={rgb, 255:red, 0; green, 0; blue, 0 }  ][line width=0.75]      (0, 0) circle [x radius= 3.35, y radius= 3.35]   ;
\draw    (398.17,88.26) .. controls (395.93,89.01) and (394.44,88.26) .. (393.69,86.03) .. controls (392.95,83.79) and (391.46,83.04) .. (389.22,83.79) .. controls (386.99,84.54) and (385.5,83.79) .. (384.75,81.56) .. controls (384,79.33) and (382.5,78.58) .. (380.27,79.33) .. controls (378.03,80.08) and (376.54,79.33) .. (375.8,77.09) .. controls (375.05,74.86) and (373.56,74.11) .. (371.33,74.86) .. controls (369.1,75.61) and (367.6,74.86) .. (366.85,72.63) -- (362.44,70.42) -- (362.44,70.42) ;
\draw  [dash pattern={on 0.84pt off 2.51pt}]  (509.19,107.27) -- (509.17,88.26) ;
\draw [shift={(509.17,88.26)}, rotate = 269.92] [color={rgb, 255:red, 0; green, 0; blue, 0 }  ][fill={rgb, 255:red, 0; green, 0; blue, 0 }  ][line width=0.75]      (0, 0) circle [x radius= 3.35, y radius= 3.35]   ;
\draw [shift={(509.19,107.27)}, rotate = 269.92] [color={rgb, 255:red, 0; green, 0; blue, 0 }  ][fill={rgb, 255:red, 0; green, 0; blue, 0 }  ][line width=0.75]      (0, 0) circle [x radius= 3.35, y radius= 3.35]   ;
\draw  [dash pattern={on 0.84pt off 2.51pt}]  (461.98,107.33) -- (461.96,88.7) ;
\draw [shift={(461.96,88.7)}, rotate = 269.92] [color={rgb, 255:red, 0; green, 0; blue, 0 }  ][fill={rgb, 255:red, 0; green, 0; blue, 0 }  ][line width=0.75]      (0, 0) circle [x radius= 3.35, y radius= 3.35]   ;
\draw [shift={(461.98,107.33)}, rotate = 269.92] [color={rgb, 255:red, 0; green, 0; blue, 0 }  ][fill={rgb, 255:red, 0; green, 0; blue, 0 }  ][line width=0.75]      (0, 0) circle [x radius= 3.35, y radius= 3.35]   ;
\draw    (485.01,118.48) -- (509.19,107.27) ;
\draw [shift={(509.19,107.27)}, rotate = 335.12] [color={rgb, 255:red, 0; green, 0; blue, 0 }  ][fill={rgb, 255:red, 0; green, 0; blue, 0 }  ][line width=0.75]      (0, 0) circle [x radius= 3.35, y radius= 3.35]   ;
\draw [shift={(485.01,118.48)}, rotate = 335.12] [color={rgb, 255:red, 0; green, 0; blue, 0 }  ][fill={rgb, 255:red, 0; green, 0; blue, 0 }  ][line width=0.75]      (0, 0) circle [x radius= 3.35, y radius= 3.35]   ;
\draw    (485.01,118.48) -- (461.98,107.33) ;
\draw [shift={(461.98,107.33)}, rotate = 205.83] [color={rgb, 255:red, 0; green, 0; blue, 0 }  ][fill={rgb, 255:red, 0; green, 0; blue, 0 }  ][line width=0.75]      (0, 0) circle [x radius= 3.35, y radius= 3.35]   ;
\draw [shift={(485.01,118.48)}, rotate = 205.83] [color={rgb, 255:red, 0; green, 0; blue, 0 }  ][fill={rgb, 255:red, 0; green, 0; blue, 0 }  ][line width=0.75]      (0, 0) circle [x radius= 3.35, y radius= 3.35]   ;
\draw  [dash pattern={on 0.84pt off 2.51pt}]  (497.34,70.39) -- (473.14,70.42) ;
\draw [shift={(473.14,70.42)}, rotate = 179.92] [color={rgb, 255:red, 0; green, 0; blue, 0 }  ][fill={rgb, 255:red, 0; green, 0; blue, 0 }  ][line width=0.75]      (0, 0) circle [x radius= 3.35, y radius= 3.35]   ;
\draw [shift={(497.34,70.39)}, rotate = 179.92] [color={rgb, 255:red, 0; green, 0; blue, 0 }  ][fill={rgb, 255:red, 0; green, 0; blue, 0 }  ][line width=0.75]      (0, 0) circle [x radius= 3.35, y radius= 3.35]   ;
\draw    (509.19,107.27) .. controls (507.03,108.21) and (505.48,107.6) .. (504.54,105.44) .. controls (503.6,103.28) and (502.05,102.67) .. (499.89,103.61) .. controls (497.73,104.55) and (496.18,103.94) .. (495.23,101.78) .. controls (494.29,99.62) and (492.74,99.01) .. (490.58,99.95) .. controls (488.42,100.89) and (486.87,100.28) .. (485.93,98.12) .. controls (484.98,95.96) and (483.43,95.35) .. (481.27,96.29) .. controls (479.11,97.23) and (477.56,96.62) .. (476.62,94.46) .. controls (475.68,92.3) and (474.13,91.69) .. (471.97,92.63) .. controls (469.81,93.58) and (468.26,92.97) .. (467.31,90.81) .. controls (466.37,88.65) and (464.82,88.04) .. (462.66,88.98) -- (461.96,88.7) -- (461.96,88.7) ;
\draw    (461.96,88.7) -- (473.14,70.42) ;
\draw [shift={(473.14,70.42)}, rotate = 301.47] [color={rgb, 255:red, 0; green, 0; blue, 0 }  ][fill={rgb, 255:red, 0; green, 0; blue, 0 }  ][line width=0.75]      (0, 0) circle [x radius= 3.35, y radius= 3.35]   ;
\draw [shift={(461.96,88.7)}, rotate = 301.47] [color={rgb, 255:red, 0; green, 0; blue, 0 }  ][fill={rgb, 255:red, 0; green, 0; blue, 0 }  ][line width=0.75]      (0, 0) circle [x radius= 3.35, y radius= 3.35]   ;
\draw    (509.17,88.26) -- (497.34,70.39) ;
\draw [shift={(497.34,70.39)}, rotate = 236.5] [color={rgb, 255:red, 0; green, 0; blue, 0 }  ][fill={rgb, 255:red, 0; green, 0; blue, 0 }  ][line width=0.75]      (0, 0) circle [x radius= 3.35, y radius= 3.35]   ;
\draw [shift={(509.17,88.26)}, rotate = 236.5] [color={rgb, 255:red, 0; green, 0; blue, 0 }  ][fill={rgb, 255:red, 0; green, 0; blue, 0 }  ][line width=0.75]      (0, 0) circle [x radius= 3.35, y radius= 3.35]   ;
\draw    (509.17,88.26) .. controls (506.93,89.01) and (505.44,88.26) .. (504.69,86.03) .. controls (503.95,83.79) and (502.46,83.04) .. (500.22,83.79) .. controls (497.99,84.54) and (496.5,83.79) .. (495.75,81.56) .. controls (495,79.33) and (493.5,78.58) .. (491.27,79.33) .. controls (489.03,80.08) and (487.54,79.33) .. (486.8,77.09) .. controls (486.05,74.86) and (484.56,74.11) .. (482.33,74.86) .. controls (480.1,75.61) and (478.6,74.86) .. (477.85,72.63) -- (473.44,70.42) -- (473.44,70.42) ;
\draw    (162.34,71.39) .. controls (198,75) and (218,87) .. (238.98,108.33) ;
\draw    (274.34,71.39) .. controls (310,75) and (330,87) .. (350.98,108.33) ;
\draw    (386.34,70.39) .. controls (422,74) and (442,86) .. (462.98,107.33) ;
\draw    (126.98,108.33) .. controls (125.31,110) and (123.65,110.01) .. (121.98,108.34) .. controls (120.31,106.68) and (118.64,106.69) .. (116.98,108.36) .. controls (115.31,110.03) and (113.65,110.03) .. (111.98,108.37) .. controls (110.31,106.71) and (108.65,106.71) .. (106.98,108.38) .. controls (105.31,110.05) and (103.65,110.05) .. (101.98,108.39) .. controls (100.31,106.73) and (98.65,106.73) .. (96.98,108.4) .. controls (95.32,110.07) and (93.65,110.08) .. (91.98,108.42) .. controls (90.31,106.76) and (88.65,106.76) .. (86.98,108.43) .. controls (85.31,110.1) and (83.65,110.1) .. (81.98,108.44) .. controls (80.31,106.78) and (78.65,106.78) .. (76.98,108.45) .. controls (75.31,110.12) and (73.65,110.12) .. (71.98,108.46) .. controls (70.31,106.8) and (68.65,106.8) .. (66.98,108.47) .. controls (65.32,110.14) and (63.65,110.15) .. (61.98,108.49) .. controls (60.31,106.83) and (58.65,106.83) .. (56.98,108.5) -- (56,108.5) -- (56,108.5) ;
\draw    (568.32,70.22) .. controls (566.65,71.89) and (564.99,71.9) .. (563.32,70.23) .. controls (561.65,68.57) and (559.98,68.58) .. (558.32,70.25) .. controls (556.65,71.92) and (554.99,71.92) .. (553.32,70.26) .. controls (551.65,68.6) and (549.99,68.6) .. (548.32,70.27) .. controls (546.65,71.94) and (544.99,71.94) .. (543.32,70.28) .. controls (541.65,68.62) and (539.99,68.62) .. (538.32,70.29) .. controls (536.66,71.96) and (534.99,71.97) .. (533.32,70.31) .. controls (531.65,68.65) and (529.99,68.65) .. (528.32,70.32) .. controls (526.65,71.99) and (524.99,71.99) .. (523.32,70.33) .. controls (521.65,68.67) and (519.99,68.67) .. (518.32,70.34) .. controls (516.65,72.01) and (514.99,72.01) .. (513.32,70.35) .. controls (511.65,68.69) and (509.99,68.69) .. (508.32,70.36) .. controls (506.66,72.03) and (504.99,72.04) .. (503.32,70.38) .. controls (501.65,68.72) and (499.99,68.72) .. (498.32,70.39) -- (497.34,70.39) -- (497.34,70.39) ;

\draw (138.58,124.65) node [anchor=north west][inner sep=0.75pt]  [rotate=-359.39]  {$g_{1} v$};
\draw (98.64,110.13) node [anchor=north west][inner sep=0.75pt]  [rotate=-0.98]  {$g_{1} x_{0}$};
\draw (166.92,52.82) node [anchor=north west][inner sep=0.75pt]  [rotate=-359.14]  {$g_{1} y_{0}$};
\draw (250.58,124.65) node [anchor=north west][inner sep=0.75pt]  [rotate=-359.39]  {$g_{2} v$};
\draw (207.64,107.13) node [anchor=north west][inner sep=0.75pt]  [rotate=-0.98]  {$g_{2} x_{0}$};
\draw (278.92,51.82) node [anchor=north west][inner sep=0.75pt]  [rotate=-359.14]  {$g_{2} y_{0}$};
\draw (362.58,123.65) node [anchor=north west][inner sep=0.75pt]  [rotate=-359.39]  {$g_{3} v$};
\draw (319.64,107.13) node [anchor=north west][inner sep=0.75pt]  [rotate=-0.98]  {$g_{3} x_{0}$};
\draw (390.92,51.82) node [anchor=north west][inner sep=0.75pt]  [rotate=-359.14]  {$g_{3} y_{0}$};
\draw (473.58,123.65) node [anchor=north west][inner sep=0.75pt]  [rotate=-359.39]  {$g_{4} v$};
\draw (431.64,107.13) node [anchor=north west][inner sep=0.75pt]  [rotate=-0.98]  {$g_{4} x_{0}$};
\draw (498.92,51.82) node [anchor=north west][inner sep=0.75pt]  [rotate=-359.14]  {$g_{4} y_{0}$};
\end{tikzpicture}
    \caption{This picture shows the path $1234$ in the auxiliary digraph, which then corresponds to a path of absorbers in the original graph. These absorbers combined can incorporate any subset of $\{g_iv:i\in [4]\}$ to the final path. In this picture, a path is shown incorporating $g_1v$ and $g_4v$ while avoiding $g_2v$ and $g_3v$.}
    \label{fig:2}
\end{figure}

Up to this point, we have a path $P$ in $\mathrm{Cay}_G(S)$ such that all the non-covered elements belong to $R_v=\{g_iv:i\in [m]\}$. As each $g_iv$ can be incorporated into $P$ using the absorber $g_iA$ (see Figure~\ref{fig:1}), we can incorporate every leftover element into the path, thus finding a Hamiltonian path in $\mathrm{Cay}_G(S)$. For the general case, when we want to rather find $O(n^\varepsilon)$ vertex-disjoint paths with specified endpoints, the proof is almost identical, the only difference being that we have to connect the endpoints appropriately to the path forest. 

\section{Preliminaries}
\subsection{Probabilistic tools}
\begin{lemma}[Chernoff's Bound~{\cite{JLR2000}}]\label{lemma:chernoff}
Let $X$ be either a binomial random variable or a hypergeometric random variable. Then, for all $0<\eps\le 3/2$,
\[\mathbb{P}\left(\big|X-\mathbb E[X]\big|\ge \eps\mathbb E[X]\right)\le 2\exp(-\eps^2\mathbb{E}[X]/3).\]
\end{lemma}
\begin{lemma}[McDiarmid's Inequality \cite{McDiarmid1989BoundedDifferences}]\label{lemma:mcdiarmid}
Let $X_1,\ldots,X_m$ be independent random variables taking values in a set $\Omega$. Let $c_1,\ldots,c_m\geq 0$ and suppose $f:\Omega^m\to\mathbb{R}$ is a function such that for every $i\in[m]$ and every $x_1,\ldots,x_m,x_i'\in\Omega$, we have $\Big|f(x_1,\ldots,x_i,\ldots,x_m)-f(x_1,\ldots,x_i',\ldots,x_m)\Big|\leq c_i$.
Then, for all $t>0$,
\[\mathbb{P}\left(\Big|f(X_1,\ldots,X_m)-\mathbb{E}[f(X_1,\ldots,X_m)]\Big|\geq t\right)\leq2\exp\left(\frac{-2t^2}{\sum_{i=1}^mc_i^2}\right).\]
\end{lemma}
\subsection{Cayley graphs}We will use the following basic facts about Cayley graphs.

\begin{lemma}\label{lem:largematching} Let $G$ be a group, let $S$ be a generating set, and let $H$ be a subgroup. If there exists an edge between the vertices of $g_1H$ and $g_2H$ in $\mathrm{Cay}_G(S)$, then there exists a matching of size at least $|H|^2/|G|$ between the vertices of $g_1H$ and $g_2H$.
\end{lemma}
\begin{proof}
Suppose $g_1 h_0 s \in g_2 H$ for some $h_0 \in H$, $s \in S$. Right multiplication by~$s$ is injective, so the set of edges $\{(g_1 h,\, g_1 h s) : h \in H,\; g_1 h s \in g_2 H\}$ forms a matching between $g_1 H$ and $g_2 H$. The size of this matching is the number of $h \in H$ such that $g_1 h s \in g_2 H$.

Since $g_1 h_0 s \in g_2 H$, we have $g_1 h s \in g_2 H$ if and only if $(g_1 h_0 s)^{-1}(g_1 h s) = s^{-1} h_0^{-1} h\, s \in H$, i.e.\ $h_0^{-1} h \in s H s^{-1}$. As $h_0, h \in H$, this is equivalent to $h_0^{-1} h \in H \cap s H s^{-1}$, and so there are exactly $|H \cap s H s^{-1}|$ such elements $h$.

Then, since $H$ and $sHs^{-1}$ are both subgroups of order $|H|$, the product formula gives
\[
|H \cdot s H s^{-1}| = \frac{|H|^2}{|H \cap s H s^{-1}|} \leq |G|,
\]
and hence $|H \cap s H s^{-1}| \geq |H|^2 / |G|$.
\end{proof}


We also use the following well-known fact about bipartite Cayley graphs.
\begin{lemma}\label{fact:bipartite}
    In a connected bipartite Cayley graph, one of the parts corresponds to an index-$2$ normal subgroup and thus the other part is the non-trivial right (or left) coset.
\end{lemma}
\begin{proof}
Since $\mathrm{Cay}(G,S)$ is connected and bipartite, the map $\phi\colon G \to \mathbb{Z}/2\mathbb{Z}$ assigning to each element the parity of its distance from~$e$ is a surjective homomorphism. Indeed, because of translation invariance we have that the distance $d(e,x)$ is the same as $d(y,xy)$, for the identity $e$ and $x,y\in G.$ By concatenating a shortest path from $e$ to $y$ with the translated shortest path from $y$ to $xy$, we obtain a walk from $e$ to $xy$ of length $d(e,x)+d(e,y)$. In a bipartite graph, all walks between two vertices have the same parity, so the parity of the distance $d(e,xy)$ must equal the parity of this concatenated walk: $\phi (xy)\equiv d(e,x)+d(e,y)\equiv \phi(x)+\phi(y)\quad (mod\,2)$.

Thus, $\phi$ is a homomorphism, so $\phi^{-1}(0)$ is a normal subgroup of index~$2$ coinciding with the part containing~$e$.
\end{proof}

\subsection{Random translates}
The next lemma states that a set formed by taking sufficiently many random translates of a set $F$ in a Cayley graph behaves similarly to a random subset of the same size.

\begin{lemma}\label{lem:randomtranslates}Let $n\in\mathbb N$ be sufficiently large and let  $m,k\in\mathbb N$ and $\sigma\in [0,1]$ satisfy $2\sigma km^{3/2}\le n\sqrt{\log n} $. Let $G$ be a finite group with $n$ elements, let $S\subset G$ be a generating set with $|S|=\sigma n$. Let $F\subset G$ be any subset with $|F|\le k$. Suppose that $g_1,\ldots, g_m\in G$ are chosen independently and uniformly at random. Then, for any  $x\in G$, 
\[\mathbb P\left(\big||N_{\mathrm{Cay}_G(S)}(x)\cap (g_1F\cup\ldots \cup g_mF)|-\sigma m|F|\big|>10\sqrt{m|F|^2\log n}\right)\le 2n^{-10}.\]
\end{lemma}
\begin{proof}
    Let us fix $F\subset G$ with $|F|\le k$ and note that, for any $x,g\in G$, we have at most $|F|$ many elements $s\in G$ which translate $x$ to $gF$, i.e. for which $xs\in gF$. Thus we have that for all $x\in G$,
    \begin{equation}\label{eq:prob:translates}
        \mathbb{P}(x\in g_1F\cup\ldots\cup g_mF)=1-\prod_{i\in [m]}\mathbb{P}(x\not\in g_i F)=1-(1-|F|/n)^m=\frac{|F|m}{n}\pm \frac{10 m^2|F|^2}{n^2},
    \end{equation}
    as long as $n$ is large enough. Then, letting $X=|N_{\mathrm{Cay}_G(S)}(x)\cap (g_1F\cup\ldots \cup g_mF)|$ we have
\[\mathbb E[X]=m\sigma |F|\pm \frac{10\sigma m^2|F|^2}{n}.\]
Note that $2\sigma km^{3/2}\le n\sqrt{\log n}$ implies $\sqrt{m|F|^2\log n}\ge {2\sigma m^2|F|^2}/{n}$, so McDiarmid's inequality (Lemma~\ref{lemma:mcdiarmid}), for the function $f(g_1,\ldots, g_m)=|N_{\mathrm{Cay}_G(S)}(x)\cap (g_1F\cup\ldots \cup g_mF)\big|$, $t=5\sqrt{m|F|^2\log n}$, and $c_i=|F|$, gives
\begin{eqnarray*}
 \mathbb P\left(\big||N_{\mathrm{Cay}_G(S)}(x)\cap (g_1F\cup\ldots \cup g_mF)\big|-\sigma m|F||>10\sqrt{m|F|^2\log n}\right) &\le & \mathbb P\left(|X-\mathbb E[X]|>t\right)\\
 &\le& 2\exp\left(-\tfrac{100m|F|^2\log n}{m|F|^2}\right)\\
 &\le&2n^{-10},
\end{eqnarray*}
as desired.
\end{proof}
\subsection{Linear forests in almost-regular graphs}
 A directed graph is $d$-regular if every vertex has exactly $d$ in-neighbours and $d$ out-neighbours (in and out-neighbourhoods of vertices are allowed to overlap). Also, we say a graph is $(d\pm d')$-regular if all degrees belong to $[d-d',d+d']$, and similarly for directed graphs. Here we are interested in finding a \textit{linear forest} or \textit{path forest} in almost-regular graphs with as few components as possible, where a linear forest is a forest whose components are all paths. For this, we use the following result. 

\begin{theorem}[\cite{christoph2025cycle}]\label{thm:magnant-martin}Every $n$-vertex $d$-regular (directed) graph has a spanning linear forest with $O(n\log d/d)$ (directed) paths that can be found in polynomial time.
\end{theorem}

Compared to the bound given above, \cite{christoph2025new} gives a sharper bound that removes the logarithmic factor, but the result from~\cite{christoph2025cycle} works for directed graphs and is algorithmic as well.
\par The following corollary of the above result will be convenient for our application.

\begin{lemma}\label{lem:directed:magnant} Let $G$ be an $n$-vertex $(d\pm d')$-regular (directed) graph where $d'\leq d$. Then, $G$ has a spanning linear forest with $O(\frac{n\log d}{d}+\frac{nd'}{d})$ (directed) paths.
\end{lemma}
\begin{proof}
    We prove the result in the directed case (undirected graphs may be viewed as directed graphs where every edge points in both directions). We wish to construct some regular digraph $G'$ for which $G$ is an induced subgraph of $G'$. If $G'$ does not have too many more vertices than $G$, a spanning linear forest of $G'$ with $k$ paths then implies the existence of a spanning linear forest of $G$ with $k+\ell$ paths, where $\ell=|V(G')\setminus V(G)|$. 
    \par As $G$ is $(d\pm d')$-regular, $G$ is an induced subgraph of some digraph $G'$ which is $(d+d')$-regular whilst containing only $O(nd'/d)$ more vertices than $G$ (this is a corollary of the main result of~\cite{gorska2009inducing}, and the undirected version of such a statement goes back to classical work of Erd\H{o}s and Kelly~\cite{erdos1963minimal}). Using Theorem~\ref{thm:magnant-martin} we get a partition of $G'$ into at most $$O\left(\frac{|V(G')|\log (d'+d)}{d'+d}\right)=O\left(\frac{|V(G')|\log d}{d}\right)$$ paths (using $d'\leq d$). Deleting all vertices in $V(G')\setminus V(G)$, we obtain a spanning path forest of $G$ of the desired form.
\end{proof}
\section{ Proof of Theorem~\ref{thm:mainthm}}\label{section:skeleton} In this section, we derive our main result (Theorem~\ref{thm:mainthm}) from \Cref{thm:maintechnical} (the latter is proved in the next section). We begin by finding the connecting structure we require across the left cosets. 
\begin{definition}[Auxiliary graph]
Let $G$ be a group, $H$ be a subgroup of $G$, and $S \subseteq G$ be a subset of elements. Let $G/H = \{gH \mid g \in G\}$ denote the set of left cosets of $H$ in $G$. The auxiliary graph $\mathrm{Aux}_{G/H}(S)$ is the graph with vertex set $ G/H$ such that two distinct cosets $C_1, C_2 \in G/H$ are adjacent if and only if there exist elements $c_1 \in C_1$ and $c_2 \in C_2$ such that $(c_1, c_2)$ is an edge in the Cayley graph $\mathrm{Cay}_G(S)$.
\end{definition}
\begin{lemma}\label{lem: transitive}
    $\mathrm{Aux}_{G/H}(S)$ is a vertex-transitive graph. In particular, it is regular. 
\end{lemma}
\begin{proof}
    For each $g \in G$, the map $xH \mapsto gxH$ is a well-defined bijection on $G/H$ that preserves adjacency (if $c_1 s = c_2$ witnesses $C_1 \sim C_2$, then $gc_1 \cdot s = gc_2$ witnesses $gC_1 \sim gC_2$), and $G$ acts transitively on $G/H$.
\end{proof}

\begin{definition}Given a subgroup $H$, we call a matching $M$ in $\mathrm{Cay}_G(S)$ an \textit{$H$-skeleton} if none of its edges have both endpoints inside any single left coset $gH$, if $V(M)$ can be enumerated as $x_0,x_1,\ldots, x_k$ where $x_{2i}\sim x_{2i+1}$ are the matched edges, and furthermore, the sequence of left cosets containing $x_0,x_1,\ldots, x_k$ forms an Euler tour of $\mathrm{Aux}_{G/H}(S)$.
\end{definition}
Suppose that $x_0,x_1,\ldots, x_k$ is a skeleton and $\mathrm{Cay}_G(S)$ is partitioned into paths $P_0,P_1,\ldots, P_{k/2}$ where $P_i$ has endpoints $(x_{2i+1},x_{2i+2})$ (indices are to be interpreted modulo $k+1$). Observe that $M\cup \bigcup P_i$ forms a Hamilton cycle of $\mathrm{Cay}_G(S)$ as
 \[x_0\xrightarrow{M}x_1\xrightarrow{P_0}x_2\xrightarrow{M}x_3\xrightarrow{P_1}x_4\xrightarrow{} \cdots \xrightarrow{P_{k/2 -1}}x_{k-1}\xrightarrow{M}x_k\xrightarrow{P_{k/2}} x_0.\]

\begin{lemma}\label{lem:skeletonsexist}
    Let $\mathrm{Cay}_G(S)$ be a connected Cayley graph and let $H$ be a subgroup of $G$. If $|H|>2|G|^{2/3}$, then $\mathrm{Cay}_G(S)$ has an $H$-skeleton that uses exactly $2\ell$ vertices from any left coset $gH$, for some $\ell\leq |G|/|H|$. 
\end{lemma}
\begin{proof}
    Consider the directed graph $R$ obtained by directing each edge of $\mathrm{Aux}_{G/H}(S)$ in both directions. As $R$ has even degrees, it has an Euler tour, say $T$. Also note that $R$ is regular (by Lemma~\ref{lem: transitive}), so the tour visits each left coset exactly $2\ell$ times, where $\ell\leq |G|/|H|$ is the degree of $\mathrm{Aux}_{G/H}(S)$.
    \par Recall Lemma~\ref{lem:largematching}, which establishes that if $C_1\sim C_2$ is an edge in $\mathrm{Aux}_{G/H}(S)$, $\mathrm{Cay}_G(S)$ has a matching of size $|H|^2/|G|$ between $C_1$ and $C_2$. As $|H|^2/|G|\geq 2 |G|/|H|$ by assumption, we can `lift' $T$ into a matching in $\mathrm{Cay}_G(S)$, meaning that we can injectively map edges $e$ of $T$ into disjoint edges of $\mathrm{Cay}_G(S)$ whose endpoints are between the left cosets that are the endpoints of $e$. It is clear the image of this map is a matching with exactly $2\ell $ vertices taken from any left coset. Furthermore, it is an $H$-skeleton, as $T$ is a Euler tour. 
\end{proof}
We can now prove the main theorem, subject to the technical theorem proved in the later section. 
\begin{proof}[Proof of Theorem~\ref{thm:mainthm} via Theorem~\ref{thm:maintechnical}] Let $S$ be a subset of $\sigma n$ elements of a $n$-element group $G$,  where $\sigma \geq n^{-c}$ and $c= 1/200$. By \Cref{prop:RobustCayley-general}, with parameters $\sigma$ and $\varepsilon=1/2$, there is a subgroup $H$ of $G$ with $|S\cap H|\geq |S|/2$, and such that $\text{Cay}_H(S\cap H)$ has no $\sigma^3/2000$-sparse cuts. If $\text{Cay}_H(S\cap H)$ is bipartite, let $H'$ be the index-$2$ subgroup of $H$ as guaranteed by Lemma~\ref{fact:bipartite}, otherwise, let $H'=H$. Let $M$ be a $H'$-skeleton, as guaranteed by Lemma~\ref{lem:skeletonsexist}, enumerated as $x_0,\ldots, x_k$. Invoke  Theorem~\ref{thm:maintechnical} with $(a_i,b_i):=(x_{2i-1}, x_{2i})$ (indices interpreted modulo $k+1$). Note that it follows from the definition of an $H'$-skeleton that the number of $\{a_i,b_i\}$ contained in any $H'$-coset is the same, as the number of times an Euler tour enters and leaves a left coset is the same, so the application is valid. We then find the desired Hamilton cycle (recall the remark after the definition of a skeleton). 
\end{proof}

\section{Robust expansion in graphs and digraphs}\label{sec:robust expansion}
The goal of this section is thus to show that the property of having no sparse cuts implies a robust notion of connectivity, which, in particular, implies that any two distinct vertices can be connected through random subsets via many vertex-disjoint paths. To do so, we use the notion of \textit{robust expansion} as systematically studied by K\"uhn, Osthus, and Treglown~\cite{kuhn2010hamiltonian}, which has been used since then to tackle several problems in extremal combinatorics (see e.g.~\cite{kuhn2010hamiltonian,kuhn2015robust,towards-graham}). The crucial definition is the following.  
\begin{definition*}
    Let $G$ be an $n$-vertex graph. For $U \subseteq V(G)$ and $\nu > 0$, we define the \emph{$\nu$-robust neighbourhood} of $U$ in $G$ to be the set
    \[RN_{\nu,G}(U) := \{v \in V(G) : |N(v) \cap U| \geq \nu n\}.\]
    We say that $G$ is a \emph{robust $(\nu, \tau)$-expander} if every $U \subseteq V(G)$ with $\tau n \leq |U| \leq (1-\tau)n$ satisfies
    \begin{equation}\label{eq:robust-expansion}|RN_{\nu, G}(U) \setminus U| \geq \nu n.\end{equation}

Similarly, if $G$ is a directed graph, one can define  the \textit{$\nu$-robust-out-neighbourhood} as 
\[RN^+_{\nu,G}(U) := \{v \in V(G) : |N(v)^- \cap U| \geq \nu n\}.\]
and say that $G$ is a \textit{robust $(\nu,\tau)$-out-expander} if~\eqref{eq:robust-expansion} holds with $RN_{\nu,G}^+(U)$ in place of $RN_{\nu,G}(U)$. \end{definition*}

Observe that if $G$ is a robust $(\nu, \tau)$-expander graph, then the directed graph $D_G$ obtained by replacing each edge of $G$ with the two edges joining its endpoints in both directions is a robust $(\nu, \tau)$-out-expander. Moreover, a directed path in $D_G$ corresponds to a path in $G$. Thus, any result about directed paths in robust out-expanders also applies to the non-directed setting.

The following simple lemma connects the property of having no sparse cuts to having good robust expansion properties.

\begin{lemma}[Proposition 5.4~\cite{bedert2025graham}]\label{lem:sparse-cut}
    Let $0\le \tau \le 1/2$ and $0 \leq \zeta \leq 1$.  Then any graph $H$ with no $\zeta$-sparse cuts is a $(\zeta \tau/8,\tau)$-robust-expander. 
\end{lemma}

The next two lemmas state that in a graph with no sparse cuts, we can robustly connect any pair of distinct vertices with paths of bounded length whose internal vertices belong to a random subset. 
\begin{lemma}\label{lem:connecting:translates}
    Let $n$ be a positive integer, and let $0<\nu, \tau, \alpha\leq 1$ and $m\in\mathbb N$ satisfy $\nu + \tau \leq \alpha$, $1\le m\le n/100$ and $m\nu^2  \geq 144\log n$. Set $\beta=\nu m/100n$. Let $G$ be an $n$-vertex graph and suppose $G$ is a robust $(\nu,\tau)$-expander with minimum degree at least $\alpha n$. Let $v_1,\ldots, v_m\in V(G)$ be chosen independently and uniformly at random, and let $V_0=\{v_i:i\in [m]\}$. Then, with probability at least $1 - n^{-1}$, the following holds:
   \begin{itemize}
       
        \item For any two distinct vertices $u,v \in V(G)$, and for any vertex subset $V_1 \subseteq V_0$ of size at most $\beta n$, there exists a path of length at most $\nu^{-1} + 1$ from $u$ to $v$ in $G$ whose internal vertices are in $V_0\setminus V_1$. 
    \end{itemize}
\end{lemma}

    \begin{proof}                                                                                                                             
  For a vertex $u$, define $R_1(u)=N_G(u)$ and $R_{i+1}(u)=R_i(u)\cup RN_{\nu,G}(R_i(u))$.
  Since $G$ is a robust $(\nu,\tau)$-expander with minimum degree at least $\alpha n\ge (\nu+\tau)n$,                                  
  we have $|R_1(u)|\ge\alpha n$ and $|R_{i+1}(u)|\ge|R_i(u)|+\nu n$
  as long as $|R_i(u)|\le(1-\tau)n$.
  So there exists $t\le\nu^{-1}$ with $|R_t(u)|>(1-\tau)n$, and since every vertex
  has degree $\ge\alpha n$ and $|V(G)\setminus R_t(u)|<\tau n$,
  every vertex of $G$ has at least $(\alpha-\tau)n\ge\nu n$ neighbours in $R_t(u)$.
  By construction, every $z\in R_{i+1}(u)\setminus R_i(u)$ satisfies
  $|N(z)\cap R_i(u)|\ge\nu n$.

  We now show that with probability $\ge 1-n^{-1}$, the following holds for
  \emph{all} $u,z\in V(G)$ and $1\le i\le t(u)$:
  \[
  |N(z)\cap R_i(u)|\ge\nu n\;\implies\;|N(z)\cap R_i(u)\cap V_0|\ge\tfrac{\nu m}{4}.\tag{$\star$}
  \]
  Fix a triple $(u,z,i)$ with $|N(z)\cap R_i(u)|\ge\nu n$
  and set $X=|N(z)\cap R_i(u)\cap V_0|$.
  As each vertex of $G$ belongs to $V_0$ with probability
  $1-(1-1/n)^m\ge m/(2n)$ (using $m\le n/100$),
  we have $\mathbb E[X]\ge\nu m/2$.
  Changing a single~$g_j$ alters~$X$ by at most~$1$,
  so McDiarmid's inequality (Lemma~\ref{lemma:mcdiarmid}) gives
  $\mathbb P(X<\nu m/4)\le\exp(-\nu^2 m/8)$.
  A union bound over at most $n^2\nu^{-1}$ triples yields failure probability
  at most $n^2\nu^{-1}\exp(-\nu^2 m/8)\le n^{-1}$,
  using $\nu^2 m\ge 144\log n$ and $\nu^{-1}\le n$.

  Condition on~($\star$) holding, and fix $u,v\in V(G)$ and $V_1\subseteq V_0$
  with $|V_1|\le\beta n$.
  Note that
  \[
  \tfrac{\nu m}{4}>\beta n+\nu^{-1}
  =\tfrac{\nu m}{100}+\nu^{-1},
  \]
  which holds since $\nu^2 m\ge 144$.
  We build a path from $v$ back to $u$ greedily.

  Since $v$ has $\ge\nu n$ neighbours in $R_t(u)$,
  by~($\star$) we can pick
  $w_1\in N(v)\cap R_t(u)\cap V_0\setminus V_1$.
  If $w_1\in R_1(u)=N(u)$, the path $u$-$w_1$-$v$ has length~$2$
  and we are done.
  Otherwise $w_1\in R_{j_1}(u)\setminus R_{j_1-1}(u)$ for some $j_1\ge 2$,
  so $|N(w_1)\cap R_{j_1-1}(u)|\ge\nu n$,
  and we pick $w_2\in N(w_1)\cap R_{j_1-1}(u)\cap V_0\setminus(V_1\cup\{w_1\})$.

  In general, at step~$s$ we pick $w_s$ from
  $N(w_{s-1})\cap R_{j_{s-1}-1}(u)\cap V_0$,
  avoiding $V_1$ and the previously chosen $w_1,\ldots,w_{s-1}$.
  This is always possible:
  by~($\star$), there are $\ge\nu m/4$ candidates in~$V_0$,
  and we remove at most $|V_1|+(s-1)\le\beta n+\nu^{-1}<\nu m/4$ of them.
  The layer index $j_1>j_2>\cdots$ is strictly decreasing,
  so after at most $t\le\nu^{-1}$ steps, some $w_s$ lands in $R_1(u)=N(u)$,
  giving the path
  \[
  u\xrightarrow{} w_s\xrightarrow{}w_{s-1}\xrightarrow{}\cdots\xrightarrow{}w_1\xrightarrow{}v
  \]
  of length $s+1\le\nu^{-1}+1$,
  with all internal vertices in $V_0\setminus V_1$.
  \end{proof}

A similar result holds for $p$-random subsets, where a $p$-random subset of a set $X$ is a random set formed by including each element of $X$ with probability $p$, making all choices independently. 
\begin{lemma}\label{lem:connecting lemma}
 Let $n$ be a sufficiently large  integer, and let $\nu, \tau, \alpha, p \leq 1$ be positive and satisfying $\nu + \tau \leq \alpha$ and $p \nu^2 n \geq 144\log n$. Define $\beta := p\nu/100$. Let $G$ be an $n$-vertex directed graph, and suppose $G$ is a robust $(\nu, \tau)$-out-expander
    with $\delta^\pm(G) \geq \alpha n$. 
    
    If $V_0\subset V(G)$ is a $p$-random subset, then with probability at least $1 - 5n^{-1}$ the following holds. For any two distinct vertices $u,v \in V(G)$, and for any vertex subset $V_1 \subseteq V_0$ of size at most $\beta n$, there exists a directed path of length at most $\nu^{-1} + 1$ from $u$ to $v$ in $G$ whose internal vertices are in $V_0 \setminus V_1$.
\end{lemma}
We will omit the routine proof of Lemma~\ref{lem:connecting lemma} as it follows the same proof of Lemma~\ref{lem:connecting:translates} with minor modifications (McDiarmid's inequality still applies).

\section{Proof of Theorem~\ref{thm:maintechnical}}\label{sec:main tech}
We are now in position to prove our main technical result (Theorem~\ref{thm:maintechnical}). The proof is split into two cases depending on whether $\mathrm{Cay}_G(S)$ is bipartite or not.
\subsection{Non-bipartite case}

To recap our situation, we have a group $G$  with $n$ elements and a symmetric generating set $S\subset G$ with $|S|=\sigma n$, where $\sigma=n^{-c}$ and $0<c\le 1/200$. Moreover, we assume $\mathrm{Cay}_G(S)$ has no  $\sigma^3/2000$-sparse cuts and so, because of Lemma~\ref{lem:sparse-cut} with parameter $\tau=\sigma/100$, for $K=1600000$ we have that 
\begin{center}
     $\mathrm{Cay}_G(S)$ is a robust $(\sigma^4/K,\sigma/100)$-expander.
\end{center}
Using  Lemma~\ref{lem:connecting lemma} with $p=1$, we may assume we have the following property.\stepcounter{propcounter}
\begin{enumerate}[label=\upshape{\textbf{\Alph{propcounter}}}]
    \item\label{thm:non-bip:connecting} For any two elements $x,y\in G$ and for any $Z\subset G$ with $|Z|\le \frac{\sigma^4n}{100K}$, there is an $(x,y)$-path of length at most $K\sigma^{-4}+1$ in $\mathrm{Cay}_G(S)$ whose internal vertices avoid $Z$.
\end{enumerate}

Let $a_1,\ldots, a_k,b_1,\ldots, b_k\in G$ be fixed and pairwise distinct elements, with $k\leq 2n^{1/200}$. We will go through~\ref{RRS:1}--\ref{RRS:4} to find vertex-disjoint $(a_i,b_i)$-paths covering $G$.\par

\textbf{Step 1. Finding an absorber gadget} \par
Our first step is to construct random-like absorbers by taking random translates of a single absorber. To do so, we first need to find a single absorber and for that the first step is to find a short odd cycle.
\begin{claim}\label{thm:non-bipartite:claim:1}There is an odd cycle $C$ in $\mathrm{Cay}_G(S)$ of length at most $\sigma^{-5}$.\end{claim}

  \begin{proof}
  Let $C$ be a shortest odd cycle in $\mathrm{Cay}_G(S)$, of length~$\ell$, which exists as we are assuming that $\mathrm{Cay}_G(S)$ is not bipartite.                                                                               
  We may assume that $\ell > \sigma^{-5}$.
  Pick a sub-path $P \subseteq C$ of length
  $\lceil 2K\sigma^{-4}\rceil$,
  with endpoints $x$ and~$y$.
  Since $|V(P)| \leq \sigma^4 n/100K$ for $n$ large,
  property \ref{thm:non-bip:connecting} provides an
  $(x,y)$-path~$L$ of length $|L| \leq K\sigma^{-4}+1 < |P|$,
  whose internal vertices avoid~$V(P)$.

  Since $|L| < |P|$, the closed walk formed by traversing~$L$
  from $x$ to~$y$ and then $C \setminus P$ from~$y$ back to~$x$
  has length $|L| + (\ell - |P|) < \ell$.
  Every closed walk shorter than the odd girth is even
  (any non-simple closed walk splits at a repeated vertex
  into two shorter closed walks, reducing to cycles
  by induction).
  Therefore $|L| + (\ell - |P|)$ is even.
  Now, $|L| + |P|$ differs from this by $2|P| - \ell$,
  which is odd since $\ell$ is odd,
  so $|L| + |P|$ is odd.

  As $L$ and $P$ are internally vertex-disjoint,
  $L \cup P$ is a simple odd cycle of length
  $|L| + |P| \leq 4K\sigma^{-4} < \sigma^{-5} \leq \ell$,
  contradicting the minimality of~$C$.
  \end{proof}
  
Let $C=g_0x_0x_1\ldots x_\ell y_0y_\ell y_{\ell-1}\ldots y_1$ be a labelling of the cycle we found in Claim~\ref{thm:non-bipartite:claim:1}, where  $2\ell\le \sigma^{-5}$. For each $i\in [\ell]$ in turn, use property~\ref{thm:non-bip:connecting} to find an $(x_i,y_i)$-path $P_i$ of length at most $K\sigma^{-4}+1$ and avoiding the paths $P_1,\ldots, P_{i-1}$, $C$ and $\{a_j,b_j:j\in [k]\}$. This is possible as the total number of vertices to avoid at each stage is at most 
\[|V(C)|+2k+\ell\cdot K\sigma^{-4}\le 2K\sigma^{-9}\le \sigma^{4}n/100K,\]
as $c< 1/13$ and $n$ is sufficiently large.
Note that the graph $A=(C-g_0)\cup\bigcup_{i\in [\ell]}P_i$ is a $g_0$-absorber as $A$ is Hamiltonian by construction, has an $(x_0,y_0)$-path not containing $g_0$ and $A+g_0$ has an $(x_0,y_0)$-path containing $g_0$, exactly as in \Cref{fig:1}.

Let $F=V(A)\cup\{g_0\}$, $m=n^{0.6}$, and note that $|F|\le 2K\sigma^{-9}$ and 
\[2\sigma|F|m^{3/2}\le 4K\sigma^{-8}n^{0.9}\le n\sqrt{\log n},\]
as $c< 1/100$. Pick elements $g_1,\ldots, g_m\in G$ uniformly at random and define $R_{g_0}=\{g_ig_0:i\in [m]\}$, $R_{x_0}=\{g_ix_0:i\in [m]\}$, $R_{y_0}=\{g_iy_0:i\in [m]\}$, and $R_{F}=\cup_{i\in [m]}g_iF$. Then, Lemma~\ref{lem:randomtranslates} gives that, with probability at least $1-O(n^{-2})$,
\begin{enumerate}[label=\upshape{(\roman{enumi})}]
\item\label{eq:prob:1} $d(x,R_{a})=\sigma m\pm 2\sqrt{m\log n}$ for each $a\in\{g_0,x_0,y_0\}$, and
\item\label{eq:prob:2} $d(x,R_F)=\sigma m|F|\pm 2|F|\sqrt{m\log n}$.
\end{enumerate}
Let $X$ be the random variable counting the number of overlapping translates, that is, the number of distinct $i,j\in [m]$ such that $g_iF\cap g_jF\not=\emptyset$. As for distinct $i,j\in [m]$ and $x\in G$ we have $\mathbb P(x\in g_iF, x\in g_jF)\le (|F|/n)^2$, it follows that 
\[\mathbb E[X]\le n\binom{m}{2}\left(\frac{|F|}{n}\right)^2=o(\sqrt{m\log n})\]
as $m^{3/2}|F|^2\le 4K^2m^{3/2}\sigma^{-18}\ll n\sqrt{\log n}$ holds for $c\leq  1/200$. Therefore, Markov's inequality implies that with high probability $X\le \sqrt{m\log n}$. For $i\in [m]$, let $Y_i$ be the random variable that takes the value 1 whenever $g_iF\cap \{a_j,b_j:j\in [k]\}\not =\emptyset$, and 0 otherwise. Similarly to Lemma~\ref{lem:randomtranslates}, we have 
\[\mathbb P(Y_i=1)=\frac{|\bigcup_j(a_j\cdot F^{-1}\cup b_j\cdot F^{-1})|}{n}\leq \frac{2k|F|}{n}.\]
Letting $Y=\sum_{i\in [m]}Y_i$, we have
\[\mathbb E[Y]\le \frac{3km|F|}{n}=o(\sqrt{m\log n}),\]
as $km|F|/n\le (2n^{1/200})n^{0.6}(2K\sigma^{-9})/n=4K\sigma^{-9}n^{-0.4+1/200}$ and $c\leq 1/200$. Hence, Markov's inequality implies  that $Y\le \sqrt{m\log n}$ holds with high probability. Let $I\subset [m]$ be a maximal collection subject to $(g_iF)_{i\in I}$ being pairwise disjoint and containing no element from $a_1,\ldots, a_k,b_1,\ldots, b_k$. Let $T_{g_0}=\{g_ig_0:i\in I\}$, $T_{x_0}=\{g_ix_0:i\in I\}$, $T_{y_0}=\{g_iy_0:i\in I\}$, and $T_F=\cup_{i\in I}g_iF$. Then, using \ref{eq:prob:1} and \ref{eq:prob:2} and conditioned on the event that both $X\le \sqrt{m\log n}$ and $Y\le \sqrt{m\log n}$, we have that for every $x\in G$,\stepcounter{propcounter}
\begin{enumerate}[label=\upshape{\textbf{\Alph{propcounter}\arabic{enumi}}}]
   \item\label{translates:degree:1} $d(x,T_a)=\sigma m\pm 10\sqrt{m\log n}$ for each $a\in \{g_0,x_0,y_0\}$
    \item\label{translates:degree:2}  $d(x,T_F)=\sigma m|F|\pm 10|F|\sqrt{m\log n}$,
    \item\label{translates:degree:complement:2} $d(x,G\setminus T_F)=\left(1-\frac{|F|m}{n}\right)\sigma n\pm 20|F|\sqrt{m\log n}$,
\end{enumerate} 
Moreover, by Lemma~\ref{lem:connecting:translates} with parameters $\alpha:=\sigma$, $\nu:=\sigma^4/K$ and $\tau:=\sigma/100$ , with high probability we have the property that
\begin{enumerate}[\textbf{\Alph{propcounter}4}]
\item \label{translates:connecting:2}for any two distinct elements $u,v\in G$ and any set $Z\subset G$ of size at most $\frac{\sigma^4m}{100K}$, there is a $(u,v)$-path of length at most $K\sigma^{-4}+1$ whose internal vertices are in $\{g_ig_0:i\in [m]\}\setminus Z$. 
\end{enumerate}
Fix a choice of $g_1,\ldots, g_m\in G$ so that \ref{translates:degree:1}--\ref{translates:connecting:2} hold and, for each $i\in I$, let $A_i:=g_iA$ and note that $g_iA$ is a $g_ig_0$-absorber with vertex set $g_iV(A)$.

\textbf{Step 2. Partition $G$ into path segments.}\par

Define an auxiliary digraph $D_I$ with vertex set $I$ and an edge from $i$ to $j$ whenever $g_iy_0$ and $g_jx_0$ are adjacent in $\mathrm{Cay}_G(S)$. The crucial observation is that if $ij$ is an edge in $D_I$, then the absorbers $A_i$ and $A_j$ can be chained into a single path through $T_F$, capable of absorbing both $g_ig_0$ and $g_jg_0$.

By~\ref{translates:degree:1}, $D_I$ is $d_I\pm d_I'$-regular for $d_I=\sigma m$ and $d_I'=10\sqrt{m\log n}$. Then, Lemma~\ref{lem:directed:magnant} gives a partition of $V(D_I)$ into at most
\[m_I=O\left(\frac{|I|\log d_I}{d_I}+ \frac{|I|d_I'}{d_I}\right)=O(\sigma^{-1}\sqrt{m\log n})\]
directed paths, where we used that $\log d_I\ll d_I'$. Each directed path in $D_I$ yields a chain of absorbers --- a path through $T_F\setminus T_{g_0}$ --- giving $m_I$ path segments whose union covers $T_F\setminus T_{g_0}$.

Let $H=G\setminus (T_F\cup\{a_j,b_j:j\in [k]\})$. By~\ref{translates:degree:2}, every vertex $x\in H$ satisfies $d(x,T_F)=\sigma m|F|\pm 10|F|\sqrt{m\log n}$, and so $H$ induces a $d_H\pm d_H'$-regular graph with $d_H\ge \sigma n/2$ and $d_H'=O(|F|\sqrt{m\log n})=O(\sigma^{-9}\sqrt{m\log n})$, as $|F|\le 2K\sigma^{-9}
$. Applying Lemma~\ref{lem:directed:magnant}, we partition $H$ into at most 
\[m'=O\left(\frac{n\log d_H}{d_H}+\frac{nd_H'}{d_H}\right)=O\left(\sigma^{-10}\sqrt{m\log n}\right)\]
paths $P_1',\ldots, P_{m'}'$, with endpoints $x_i',y_i'$ for each $i\in [m']$. By possibly splitting some paths, we may assume $m'\ge k$.\par

\textbf{Step 3. Connecting into $k$ spanning paths through $T_{g_0}$.}\par

We merge all segments into $k$ paths, each with prescribed endpoints $a_j,b_j$, using connections through $T_{g_0}$ via~\ref{translates:connecting:2}. For each $j=1,\ldots, k-1$, use~\ref{translates:connecting:2} to connect $a_j$ to $x_j'$ and $y_j'$ to $b_j$, creating $k-1$ paths each containing one $H$-path. For the remaining $m_I+m'-k+1$ segments (all $m_I$ absorber chains and the $m'-k+1$ remaining $H$-paths), chain them into a single path by connecting consecutive endpoints via~\ref{translates:connecting:2}, and connect $a_k$ and $b_k$ to its two ends. In each application of~\ref{translates:connecting:2}, we avoid $\{g_ig_0:i\not\in I\}$ and all vertices previously used for connections. As $\left|\{g_ig_0:i\not\in I\}\right|\le 2\sqrt{m\log n}$ and the total number of connecting vertices is at most $(m_I+m'+k)\cdot K\sigma^{-4}=O(\sigma^{-14}\sqrt{m\log n})$, the avoided set has size 
\[O(\sigma^{-14}\sqrt{m\log n})\le \frac{\sigma^4 m}{100K},\]
as $c\le 1/200$ and $m=n^{0.6}$.  \par
\textbf{Step 4. Incorporating the leftover.}

For $j\in [k-1]$, let $P_j$ be the $(a_j,b_j)$-path just constructed, and let $\tilde{P}$ be the $(a_k,b_k)$-path. The leftover elements are contained in $\{g_ig_0:i\in I\}$. Define $I'\subset I$ as the set of indices $i\in I$ such that none of the paths $P_1,\ldots, P_{k-1},\tilde{P}$ uses $g_ig_0$. For each $i\in I'$, $\tilde{P}$ contains $A_i$ as a segment. Using that $A_i$ is a $g_ig_0$-absorber, we replace $A_i$ with a path $A_i'$ having the same endpoints as $A_i$ but containing $g_ig_0$. Calling the resulting path $P_k$, we have that $P_k$ is an $(a_k,b_k)$-path and every element of $G$ is covered by one of the paths.

\subsection{Bipartite case}\label{sec:bip}

As in the non-bipartite case, we have a group $G$ with $n$ elements and a symmetric generating set $S\subset G$ with $|S|=\sigma n$, where $\sigma=n^{-c}$ and $0<c\le 1/200$. As before, we assume $\mathrm{Cay}_G(S)$ has no  $\sigma^3/2000$-sparse cuts and so, because of Lemma~\ref{lem:sparse-cut}, for $K=1600000$ we have that
$\mathrm{Cay}_G(S)$ is a robust $(\sigma^4/K,\sigma/100)$-expander. Although we could run a similar proof as in the non-bipartite case, with slightly adjusted absorbers, in this case we are able to adopt a more direct approach. 

Recall that our definition of robust expansion means that $|
RN_{\nu,\mathrm{Cay}_G(S)}(U)\setminus U|\geq n\sigma^4/K$ for all subsets $U\subset G$ with $\tau n\leq|U|\leq (1-\tau)n$, where $\tau=\sigma/100$. This expansion is rather weak, but, since our graph is bipartite, in fact a stronger claim holds.

\begin{proposition}\label{lem:bipartite expansion}Let $n,d\in\mathbb N $ and $\nu,\tau,\zeta\in (0,1)$ satisfy $\nu(n+d) < \zeta \tau (1-\tau)n$. Let $H$ be an $n$-vertex balanced $d$-regular bipartite graph with parts $A$ and $B$, with no $\zeta$-sparse cuts. Then, for every set $X\subset A$ of size $\tau n\leq|X|\leq (1/2-\tau)n$, $|RN_{\nu,H}(X)|\geq |X|+\nu n$.\end{proposition}

\begin{proof}

Denote $R:=RN_{\nu,H}(X)$ and note that $R\subset B$. Assume for contradiction that $|R|\leq |X|+\nu n$. Let $Y = B \setminus R$. By the definition of the robust neighbourhood, for every $v \in Y$ we have $d_H(v, X) < \nu n$. Consequently, the number of edges between $X$ and $Y$ satisfies
\[e(X, Y) = \sum_{v \in Y} d_H(v, X) < |Y|\nu n.\] Let $S = X \cup R$ and consider the cut $(S, V(H)\setminus S)$. Note that $V(H) \setminus S = (A \setminus X) \cup Y$. The edges crossing the cut are exactly those between $X$ and $Y$, and those between $R$ and $A \setminus X$. Using the $d$-regularity of $H$, we have $d|R| = e(R, X) + e(R, A \setminus X)$. Similarly, the edges incident to $X$ are $d|X| = e(X, R) + e(X, Y)$.
Substituting $e(R, X) = d|X| - e(X, Y)$ into the first equation yields
\[e(R, A \setminus X) = d|R| - (d|X| - e(X, Y)) = d(|R|-|X|) + e(X, Y),\]
hence the total number of edges crossing the cut is
\[e(S, V(H)\setminus S) = e(X, Y) + e(R, A \setminus X) = 2e(X, Y) + d(|R|-|X|).\]
Using the bound $e(X, Y) < |Y|\nu n$ and the assumption $|R| - |X| \leq \nu n$, we obtain
\[e(S, V(H)\setminus S) < 2|Y|\nu n + d\nu n \leq \nu n^2 + d\nu n = \nu n (n+d),\]
where we used that $|Y| \leq |B| = n/2$. On the other hand, since $H$ has no $\zeta$-sparse cuts, we have $e(S, V(H)\setminus S) \geq \zeta |S| (n - |S|).$ 

As the function $f(x) = x(n-x)$ is minimized at the boundaries of the interval $[\tau n, n - \tau n]$, we therefore have $|S|(n-|S|) \geq \tau(1-\tau)n^2.$ Combining the lower and upper bounds gives
\[\zeta \tau(1-\tau)n^2 \leq e(S, V(H)\setminus S) < \nu n (n+d),\]
thus dividing by $n$ we require $\zeta \tau (1-\tau) n \leq \nu (n+d)$, a contradiction with our assumption.\end{proof}

Applying Proposition~\ref{lem:bipartite expansion}, with parameters $\zeta=\sigma^3/2000$, $\tau= \sigma/100$, $d=\sigma n$, and $\nu =\sigma^4/K$, we get that in $\mathrm{Cay}_G(S)$, every set $X\subset G$ (coming from only one part of the bipartition) of size $\tau n\leq|X|\leq (1/2-\tau)n$ has $|RN_{\nu,\mathrm{Cay}_G(S)}(X)|\geq |X|+\nu n$.

We will use the following result by Lo and Patel~\cite{lo2015hamilton}, which states that sufficiently strong directed robust out-expanders are Hamiltonian (Theorem 1.3 in \cite{lo2015hamilton}, with $\ell=n$). 

\begin{theorem}\label{thm:lopatel}
Let $n\in\mathbb{N}$ and let $\nu,\tau,\gamma\in(0,1)$ satisfy
$4\sqrt[13]{\log^2 n / n}<\nu\le\tau\le\gamma/16<1/16$.
Let $D$ be an $n$-vertex digraph with minimum semidegree
$\delta^0(D)\ge\gamma n$ and such that every $U \subseteq V(D)$ with $\tau n \leq |U| \leq (1-\tau)n$ satisfies
\[|RN^+_{\nu,D}(U)|\ge |U|+\nu n.\]
Then, $D$ is Hamiltonian.
\end{theorem}
\par Note that the expansion hypothesis of the above theorem is significantly stronger than robust out-expansion as defined in Section~\ref{sec:robust expansion}. In particular, the hypothesis above can never hold for a bipartite digraph, as can be seen by taking $U$ to be the larger partition class. Also note that Theorem~\ref{thm:lopatel} does not assume any regularity conditions on the underlying digraph\footnote{For weaker notions of expansion that we work with, regularity is crucial to record as Hamiltonicity may fail due to the underlying graph being an unbalanced bipartite graph.}. 
\par Theorem~\ref{thm:lopatel} has the following convenient corollary, giving Hamilton connectivity.
\begin{corollary}\label{cor:lopatel} Let $n\in\mathbb{N}$ be sufficiently large and let $\nu,\tau,\gamma\in(0,1)$ satisfy
$5\sqrt[13]{\log^2 n / n}<\nu/2\le2\tau\le\gamma/32<1/16$.
Let $D$ be an $n$-vertex digraph with minimum semidegree
$\delta^0(D)\ge\gamma n$ and such that every $U \subseteq V(D)$ with $\tau n \leq |U| \leq (1-\tau)n$ satisfies
\[|RN^+_{\nu,D}(U)|\ge |U|+\nu n.\]
Let $u,v$ be arbitrary vertices. Then, $D$ has a Hamilton directed path from $u$ to $v$.
\end{corollary}
\begin{proof} Let $D'$ be a directed graph obtained from $D$ by contracting $u$ and $v$ to a vertex, keeping the out-neighbours of $u$ and the in-neighbours of $v$. Note that for every $U\subset V(D')$ with $2\tau |V(D')|\le |U|\le (1-2\tau)|V(D')|$, we have 
\[|RN^+_{\nu/2,D'}(U)|\ge |U|-2+\nu n\ge |U|+\nu|V(D')|/2,\]
provided $n$ is large enough. Then, by Theorem~\ref{thm:lopatel} with parameters $\gamma/2,2\tau$ and $\nu/2$, there is a directed Hamilton cycle in $D'$, which implies $D$ has a Hamilton path from $u$ to $v$.
\end{proof}

Recall that our goal is to cover $G$ with paths from $a_i$ to
$b_i$, for all $i\in[k]$. Denote $R=\{a_1,\ldots,a_k,b_1,\ldots,
b_k\}$ the set of vertices in all pairs. Since $k=O(n^{1/200})$
and $|S|\ge n^{1-c}$, we can fix an element $g'\in S$ such that
$Rg'\cap R=\emptyset$. Notice that $\{(h,hg')\mid h\in H\}$ forms
a perfect matching $M$ in $G$, where each matched edge contains at most
one vertex in $R$.

Let $X\subset[k]$ be the set of indices $i$ for which
$\{a_i,b_i\}\subset H$, and $Y=[k]\setminus X$; hence $|X|=|Y|$.
Denote $B:=\{b_ig'\mid i\in X\}$ the set of vertices matched to
the $b_i$'s in $H$, and similarly $A:=\{a_jg'^{-1}\mid j\in Y\}$.
Notice that the set $R\cup\{a_ig'\mid i\in X\}\cup
\{b_jg'^{-1}\mid j\in Y\}\cup A\cup B$ is a union of edges of $M$;
call its complement in $V(G)$ the \emph{bulk}, and note that $M$
restricts to a perfect matching on the bulk.

Now form a directed graph $D$ from $G$ by contracting each edge
$(h,hg')$ lying entirely in the bulk; that is, the vertices of
$D$ are identified with such edges $e_h=(h,hg')$, the
out-neighbours of $e_h$ are all $e_f$ such that $(hg',f)$ is an
edge in $G$, and the in-neighbours are all $e_f$ such that
$(fg',h)$ is an edge in $G$. Then let $D^\ast$ be obtained from
$D$ by adding the vertices in $A\cup B$: each $b\in B\subset\bar
gH$ gets out-neighbours equal to its $G$-neighbours and no
in-neighbours, and each $a\in A\subset H$ gets in-neighbours
equal to its $G$-neighbours and no out-neighbours. All arcs of
$D^\ast$ come from genuine edges of $G$.

Observe that $|RN^+_{\nu,D^\ast}(U)|\ge
|U|+\nu n/2$ for all $U\subset V(D^\ast)$ with $\tau n\le|U|\le
(1-\tau)|V(D^\ast)|$ (since we only changed edges touching at
most $O(k)=o(n)$ vertices to obtain $D^\ast$ from $D$, and the latter graph has a stronger version of the claimed property with $\nu/2$ replaced with $\nu$ by \Cref{lem:bipartite expansion}). Fix an
arbitrary bijection $\phi:X\to Y$, and use
Lemma~\ref{lem:connecting lemma} with $p=1$ to find, one at a
time, vertex-disjoint directed paths $P_1,\ldots,P_{|X|}$ in
$D^\ast$, where $P_i$ goes from $b_ig'\in B$ to
$a_{\phi(i)}g'^{-1}\in A$ and has length at most
$2\sigma^{-1}\nu^{-1}$. This is possible because the total number
of vertices used is at most
\[
|X|\cdot 2\sigma^{-1}\nu^{-1}=o(\nu n),
\]
as $c\le 1/200$. Each $P_i$ lifts to a path in $G$ from $b_ig'$
to $a_{\phi(i)}g'^{-1}$ whose interior uses only bulk vertices,
and the lifted paths are still vertex-disjoint and avoid $R$.
This finishes the preliminary step; we now connect the pairs
$(a_i,b_i)$ by a similar procedure.

We move to the next auxiliary graph $\bar D$, obtained from $G$
as follows. First, contract each path $P_i$ into a single vertex
$\tilde P_i$: its in-neighbours are the $G$-neighbours of
$a_{\phi(i)}g'^{-1}$ and its out-neighbours are the $G$-neighbours
of $b_ig'$. This absorbs the endpoints $a_{\phi(i)}g'^{-1}\in A$
and $b_ig'\in B$, along with every $M$-edge in the interior of
$P_i$. Then contract every edge $(h,hg')$ of $M$ whose endpoints
are disjoint from $R$ and from every $V(P_i)$, in the same manner
as in $D$. Finally, for each $i\in X$ contract the edge
$(a_i,a_ig')$, and for each $j\in Y$ contract the edge
$(b_jg'^{-1},b_j)$, again in the same manner.

At this point, the $M$-edges of the form $(b_i,b_ig')$ for $i\in X$
and $(a_jg'^{-1},a_j)$ for $j\in Y$ are the only ones left
untouched: in each such edge, one endpoint ($b_ig'\in B$ or
$a_jg'^{-1}\in A$) has been absorbed into a $\tilde P_\ell$,
while the other ($b_i$ or $a_j$) has not. We add these leftover
vertices back as singletons: each $b_i\in H$ ($i\in X$) gets
in-neighbours equal to its $G$-neighbours and no out-neighbours,
and each $a_j\in\bar gH$ ($j\in Y$) gets out-neighbours equal to
its $G$-neighbours and no in-neighbours. Now every vertex of $G$
is represented by exactly one vertex of $\bar D$, and every arc of
$\bar D$ comes from a genuine edge/path of $G$.

Since we only changed edges touching at most $o(\nu n)$ vertices
compared to $D$, $\bar D$ is still an expander with essentially
the same parameters. Also, a directed path in $\bar D$ ending at the
singleton $b_i$ (for $i\in X$) lifts to a path in $G$ ending at
$b_i$, and a directed path starting at the singleton $a_j$ (for
$j\in Y$) lifts to a path in $G$ starting at $a_j$. In particular,
setting
\[
v_i:=e_{a_i}  \text{ and }  u_i:=b_i\ \text{ for }i\in X,\qquad
v_j:=a_j\text{ and } u_j:=e_{b_jg'^{-1}}\ \text{ for }j\in Y,
\]
a directed $v_i\to u_i$ path in $\bar D$ lifts to a path in $G$
from $a_i$ to $b_i$. Using Lemma~\ref{lem:connecting lemma} with
$p=1$, we find vertex-disjoint directed paths from $v_i$ to $u_i$
for $i\in[k-1]$, each of length at most $2\sigma^{-1}\nu^{-1}$.
This is possible as the number of vertices used is at most
\[
(k-1)\cdot 2\sigma^{-1}\nu^{-1}=o(
\nu n/100),
\]
as $c\le 1/200$.

It remains to connect $v_k$ and $u_k$ with a directed path that
spans all remaining vertices. Consider the digraph $\bar D'$ with
all $v_i,u_i$ ($i\in[k-1]$) and the connecting paths found
in-between deleted. By the above inequality and since $\nu\ll
\tau$, $\bar D'$ enjoys essentially the same expansion properties
as $\bar D$, so Corollary~\ref{cor:lopatel} applies to give a
Hamilton directed path from $v_k$ to $u_k$. Combined with the
previous paths, we have a disjoint collection of paths
partitioning $V(\bar D)$ with endpoints $v_i\to u_i$ for each
$i\in[k]$. Lifting each of them back to $G$ gives a spanning
linear forest whose components are paths from $a_i$ to $b_i$ for
each $i\in[k]$, completing the proof.

\begin{remark}
Strictly speaking, applying Lemma~\ref{lem:connecting lemma} and
Corollary~\ref{cor:lopatel} requires a minimum semi-degree
condition which fails at the singleton sources and sinks in
$D^\ast$ and $\bar D$, as these vertices have zero degree on one
side. This is only a technicality: on the missing side, we may
simply declare each such vertex to have every other vertex as a
neighbour. These artificial arcs do not affect the lift, since
each singleton is used only as an endpoint of one of the
constructed paths and so no artificial arc ever appears in the
interior of a path. With this convention the minimum semi-degree
hypothesis is satisfied and the lemmas apply as stated.
\end{remark}

\section*{Acknowledgements}
We are grateful to Shoham Letzter for bringing the reference \cite{lo2015hamilton} to our attention, as it yields a more concise argument for the result in Section~\ref{sec:bip} compared to an earlier version of this manuscript.

\bibliography{bib}
\end{document}